\documentclass[12pt]{iopart}

\usepackage{amsthm,amssymb,bbm,graphicx,psfrag,color,enumerate,url}
\usepackage[latin1]{inputenc}

\newtheorem{thm}{Theorem}[section]

\newtheorem{lemma}[thm]{Lemma}

\newtheorem{remark}[thm]{Remark}
\newtheorem{exa}[thm]{Example}

\def\req#1{{\rm(\ref{eq:#1})}}
\newcommand{\labeq}[1]{\label{eq:#1}}
\def\norm#1{\hspace{0.2ex} \|#1\| \hspace{0.2ex}}

\newcommand{\im}{{\rmi}}

\newcommand{\Hd}{H_\diamond}
\newcommand{\Ld}{L_\diamond}

\newcommand{\dx}[1][x]{\ensuremath{\ \rmd #1}}

\newcommand{\R}{\ensuremath{\mathbbm{R}}}
\newcommand{\C}{\ensuremath{\mathbbm{C}}}
\newcommand{\N}{\ensuremath{\mathbbm{N}}}

\newcommand{\kommentar}[1]{}

\usepackage{tikz}
\usepackage{tikz-3dplot}
\begin{document}

\title[Combining fdEIT and US-modulated EIT]{Combining frequency-difference and ultrasound modulated electrical impedance tomography\footnotemark}

% Statement of Provenance
\renewcommand{\footnoterule}{%
  \kern -3pt
  \hrule width \textwidth height 1pt
  \kern 2pt
}
\footnotetext{%
\scriptsize
This is an author-created, un-copyedited version of an article published in \emph{Inverse Problems} \textbf{31}(9), 095003, 2015.
IOP Publishing Ltd is not responsible for any errors or omissions in this version of the manuscript or any version derived from it. The Version of Record is available online at 
\url{http://dx.doi.org/10.1088/0266-5611/31/9/095003}.
}

\author{Bastian Harrach$^{1}$, Eunjung Lee$^{2}$, Marcel Ullrich$^{1}$}
\address{$^{1}$ Department of Mathematics,
University of Stuttgart, Germany}
\address{$^{2}$ Department of Computational Science and Engineering, Yonsei University, Seoul, Korea}
\ead{\mailto{harrach@math.uni-stuttgart.de}, \mailto{eunjunglee@yonsei.ac.kr}, \\ 
\mailto{marcel.ullrich@mathematik.uni-stuttgart.de}}

\begin{abstract}
Electrical impedance tomography (EIT) is highly affected by modeling errors regarding electrode positions and the shape of the imaging domain. In this work, we propose a new inclusion detection technique that is completely independent of such errors. Our new approach is based on a combination of frequency-difference and ultrasound modulated
EIT measurements.
\end{abstract}

%\begin{keywords}
%Inverse Problems, Electrical Impedance Tomography, frequency-difference EIT, ultrasound modulated EIT, monotonicity, localized potentials
%\end{keywords}


%\submitto{\IP}

%%%%%%%%%%%%%%%%%%%%%%%%%%%%%%%%%%%%%%%%%%%%%%%%%%%%%%%%%%%%%%%%%%%%%%%%%%%%
\section{Introduction}
\label{Sec:intro}

The goal of electrical impedance tomography (EIT) is to image the
conductivity inside a subject. To that end, electrodes are attached to the subject's boundary,
and one measures the voltages that are required to drive a specified static or time-harmonic current
through different combinations of the attached electrodes.
The potential advantages of EIT compared to other imaging technique are
that conductivity values are typically of a high specificity, and that EIT devices are comparatively cheap and easily portable. 

The inverse problem of reconstructing the conductivity from boundary voltage and current
measurements is known to be highly non-linear and ill-posed. The measurements are very insensitive
to changes in the conductivity values away from the electrodes.
They do, however, strongly depend on the measurement geometry, i.e., the electrode position and the shape of the
imaging domain. In most applications, it is not feasible to precisely measure the geometry, and electrodes
are frequently placed by hand. Hence, such modeling or geometry errors present a major challenge for practical EIT applications.

The main focus of EIT is often on the detection and localization of conductivity
inclusions or anomalies (e.g., material faults or pathological regions) inside an
otherwise more or less homogeneous medium. 

In this work, we propose a new measurement setup for anomaly detection and describe a
reconstruction method that is completely unaffected by geometrical modelling errors,
as it does not require knowledge of the electrode position or the shape
of the imaging domain.

The main idea of our new technique is to combine ultrasound-modulated EIT measurements
with frequency-difference EIT measurements. We focus an ultrasound wave on a small region inside the imaging domain
to alter the conductivity in the focusing region. The resulting effect on the EIT measurements is then
compared to the effect of a change in the electric current frequency.
This comparison shows whether the focusing region lies inside a conductivity anomaly or not.

To decide whether the focusing region lies inside an anomaly, our method
utilizes only the two sets of EIT measurements (with ultrasound-modulation and
after the frequency change) and the ratio of
the background conductivity before and after the frequency change.
The latter can be estimated from comparing EIT measurements
before and after the frequency change, as it is done in weighted frequency-difference
EIT (see the references below). The method can be implemented using simple monotonicity tests, i.e., the taken voltage measurements are arranged in the form of a matrix and then compared in the sense of matrix definiteness
(resp., in the idealized case of continuous boundary measurements, the measurements are interpreted as Neumann-to-Dirichlet operators and compared in the sense of definiteness of self-adjoint compact operators).

Our new method does not use any forward simulations, or explicitly known special solutions, that would depend on the geometry of the setup. It does not require any knowledge of the
electrode position or the shape of the imaging domain, and is hence completely
unaffected by modeling errors.

We give a complete proof for our method for the case of
continuous boundary data, when the measurements are given by the Neumann-to-Dirichlet-operator.
For the case of measurements on a finite number of electrodes,
we prove that the method correctly identifies the case where the focusing region lies inside
the anomaly. We also give a physical justification (in the spirit of \cite{harrach2010factorization}),
that regions outside the anomaly will correctly be identified if
enough electrodes are used for the measurements, cf.\ remark \ref{remark:electrode}.

Let us now comment on related works and the origins of our approach. For a broad overview on electrical impedance tomography see \cite{henderson1978impedance,barber1984applied,brown1987sheffield,Metherall1996,Che99,Borcea2002,borcea2003addendum,Brown2003,lionheart2003eit,holder2004electrical,holder2004electrical,bayford2006bioimpedance,adler2011electrical,uhlmann2009electrical,Nguyen2012}. For the task of anomaly detection in EIT, let us refer to Friedmann and Isakov \cite{friedman1987detection,friedman1989uniqueness} for early works, Potthast \cite{potthast2006survey} for an overview on non-iterative methods, and \cite{harrach2010exact} for the recent result that shape
information is invariant under linearization. Iterative anomaly detection methods are commonly based on level-set approaches, cf., e.g. \cite{ito2001level,tai2004survey,chung2005electrical,dorn2006level,van2006level,rahmati2012level}.
Prominent non-iterative anomaly detection methods are the Factorization Method (see
\cite{kirsch2005factorization,Han03,Geb05,Han07,Hyv07,lechleiter2008factorization,Nac07_2,GH07,hakula2009computation,harrach2009detecting,Schmitt2009,harrach2010factorization,schmitt2011factorization,haddar2013numerical,chaulet2014factorization,choi2014regularizing,barthdetecting} and the recent overviews
\cite{Kirsch_book07,hanke2011sampling,harrach2013recent}), the enclosure method (see \cite{ikehata1999draw,Bru00,ikehata2000reconstruction,ikehata2000numerical,ikehata2002regularized,ikehata2004electrical,ide2007probing,uhlmann2008reconstructing,ide2010local}),
and the recently emerging monotonicity method (see the references below).

Our new method is based on a monotonicity-based comparison of
weighted frequency-difference EIT (fdEIT) and ultrasound-modulated EIT (UMEIT) measurements. Monotonicity-based comparisons were first considered as heuristic inclusion-detection methods and numerically tested by Tamburrino and Rubinacci \cite{Tamburrino02,Tamburrino06}. Recently, the monotonicity method was rigorously justified \cite{harrach2013monotonicity} using the concept of localized potentials \cite{gebauer2008localized}. Weighted fdEIT has been introduced in order to improve
the reconstruction stability with respect to modeling errors in settings where no reference (anomaly-free) data is available, see \cite{seo2008frequency,harrach2009detecting,harrach2010factorization}.
The hybrid tomography technique UMEIT was introduced in \cite{zhang2004acousto,ammari2008electrical},\ cf. also \cite{kuchment20112d,ammari2012resolution,kocyigit2012acousto,kuchment2012mathematics,widlak2012hybrid,bal2013cauchy,monard2013inverse,bal2013levenberg,bonnetier2014note} for more works on this subject.
When the measurement geometry is known, UMEIT allows to measure the interior electrical energy of the subject by altering the conductivity with a focused ultrasound waves (cf.\ the related idea of Impedance-Acoustic Tomography \cite{GS08}, where interior energy data is obtained from measuring expansion effects caused by electrical heating).
Knowledge of this additional interior energy information eliminates the major cause of ill-posedness in the reconstruction process, which could greatly increase image resolution.
Moreover, let us mention that combinations of EIT and ultrasound have been studied 
that rely on data-fusion rather than on coupled physics, e.g., by using ultrasound images as prior information for EIT reconstructions, cf., e.g., \cite{soleimani2006electrical,steiner2008bio}.

At this point, it has to be noted, that (up to the knowledge of the authors) the idea of using focused waves in ultrasound-modulated EIT (UMEIT) yet has to be experimentally validated. The results in this work are derived under the idealistic assumption of a perfectly focused ultrasound waves that changes the
conductivity in a well-defined circular region. Of course, in reality, such a perfect focus cannot be realized, and the ultrasound wave will also affect the conductivity outside the focusing region.
Moreover, the location of the focusing region will not be known exactly but depend on the measurement geometry. It is, however, widely accepted that in typical EIT applications, conductivity contrast is much higher than ultrasound contrast, while ultrasound resolution is much higher than EIT resolution. 
Therefore we believe that techniques relying on UMEIT are worth investigating despite the current lack of practical validation.

The paper is organized as follows. In section \ref{Sec:continuous}, we start with describing the general setting of complex conductivity EIT and ultrasound modulated EIT for continuous boundary data.
Then we derive a monotonicity relation for complex conductivity EIT, and use this relation 
to develop an anomaly detection algorithm that is based on comparing EIT measurements at a non-zero
frequency with ultrasound-modulated DC measurements. Section \ref{Sec:electrodes} contains the corresponding results for a setting with finitely many electrodes using the shunt electrode model. 
In section \ref{Sec:numerics}, we illustrate our new method with two- and three-dimensional numerical results. Section \ref{Sec:conlusion_and_discussion} concludes the paper with a discussion of our results.

%%%%%%%%%%%%%%%%%%%%%%%%%%%%%%%%%%%%%%%%%%%%%%%%%%%%%%%%%%%%%%%%%%%%%%%%%%%%
\section{Continuous boundary data}
\label{Sec:continuous}

\subsection{The setting}
\label{Subsec:setting_continuous}

We start by describing the general setting of complex conductivity EIT and ultrasound modulated EIT
with continuous boundary data.
We consider a bounded imaging domain $\Omega\subset \R^n$, $n\geq 2$ with piecewise smooth boundary.
For $x\in \Omega$, let
\[
\gamma_\omega(x)=\sigma_\omega(x)+\im \omega \epsilon_\omega(x)
\]
denote the body's complex admittivity at frequency $\omega\geq 0$.
We assume that
\[
\Re (\gamma_\omega)=\sigma_\omega \in L_+^\infty(\Omega;\R), \quad \mbox{ and } \quad \Im (\gamma_\omega)=\omega\epsilon_\omega \in L^\infty(\Omega;\R),
\]
where
$\Re(\cdot)$ and $\Im(\cdot)$ denote the real and imaginary part, the subscript ``$+$'' indicates
functions with positive (essential) infima, and throughout this work all function spaces consist of complex valued
functions if not stated otherwise.

Complex EIT measurements consist of applying time-harmonic currents to the surface of the imaging domain
and measuring the resulting electric surface potential.
In the so-called continuum model (see, e.g., \cite{Che99}),
these measurements are described by the Neumann-to-Dirichlet-Operator
\[
\Lambda(\gamma_\omega):\; \Ld^2(\partial \Omega)\to \Ld^2(\partial \Omega), \quad g\mapsto u_{\gamma_\omega}^{(g)}|_{\partial \Omega},
\]
where $u_{\gamma_\omega}^{(g)}\in \Hd^1(\Omega)$ solves
\begin{equation}
\labeq{EIT}
\nabla \cdot \left(\gamma_\omega \nabla u_{\gamma_\omega}^{(g)}\right) = 0\mbox{ in } \Omega \quad \mbox{ and } \quad
\gamma_\omega \partial _\nu u_{\gamma_\omega}^{(g)}|_{\partial \Omega}=g.
\end{equation}
Here, the subspace of $L^2(\partial \Omega)$ and $H^1(\Omega)$-functions with vanishing integral mean
on $\partial \Omega$ is denoted by $\Ld^2(\partial \Omega)$ and $\Hd^1(\Omega)$, respectively.
$\nu$ is the outer normal on $\partial \Omega$.
It is well known that $\Lambda(\gamma_\omega)$ is a well-defined, linear and compact operator.

The idea of ultrasound-modulated EIT is to focus an ultrasound wave on a small part $B\subseteq \Omega$
in order to change the density of the material and thus its conductivity in $B$, cf.\ \cite{ammari2008electrical}.
A simple, very idealistic model is that the focused ultrasound wave changes the conductivity from
$\gamma_\omega$ to $\gamma_\omega(1+\beta \chi_B)$, where $\beta>0$ depends on the strength of the ultrasound wave and $\chi_B$ is the characteristic function of $B$.
Hence, ultrasound-modulated EIT measurements can be modeled as
\[
\Lambda(\gamma_\omega (1+\beta \chi_B)).
\]

In this work, we will compare measurements at a non-zero frequency $\Lambda(\gamma_\omega)$, $\omega>0$,
with ultrasound-modulated DC measurements $\Lambda(\gamma_0 (1+\beta \chi_B))$ in order to detect whether
the ultrasound modulated part $B$ lies inside a conductivity anomaly or not.

\subsection{Monotonicity results for the continuous case}

We will compare measurements in the sense of operator definiteness. Given a
bounded linear operator $A:\ \Ld^2(\partial \Omega)\to \Ld^2(\partial \Omega)$, we define its self-adjoint part by setting
\[
\Re(A):=\frac{1}{2}(A+A^*)
\]
where $A^*:\ \Ld^2(\partial \Omega)\to \Ld^2(\partial \Omega)$ is the adjoint of $A$ with respect to the inner product
of $\Ld^2(\partial \Omega)$, i.e.,
\[
\int_{\partial \Omega} \overline{g} (A h)\dx[s] = \int_{\partial \Omega} (\overline{A^*g}) h\dx[s] \quad
\mbox{ for all } g,h\in \Ld^2(\partial \Omega).
\]
Obviously, $\Re(A)$ is a self-adjoint bounded linear operator.

For two self-adjoint bounded linear operators $A,B:\ \Ld^2(\partial \Omega)\to \Ld^2(\partial \Omega)$, we write $A\leq B$ if
$B-A$ is positive semidefinite, i.e.
\[
\int_{\partial \Omega} \overline{g} A g\dx[s]\leq \int_{\partial \Omega} \overline{g} B g\dx[s] \quad \forall g\in \Ld^2(\partial \Omega).
\]
For compact operators, this is equivalent to the fact that all eigenvalues of $B-A$ are non-negative.

Note that, for all $g,h\in \Ld^2(\partial \Omega)$, the Neumann-to-Dirichlet-operator $\Lambda(\gamma_\omega)$ satisfies
\begin{eqnarray*}
\int_{\partial \Omega} \overline{g} \Lambda(\gamma_\omega) h\dx[s]
=\int_{\partial \Omega} \overline{g} u_{\gamma_\omega}^{(h)}|_{\partial \Omega} \dx[s]
= \int_\Omega \overline{\gamma_\omega \nabla u_{\gamma_\omega}^{(g)}} \cdot \nabla u_{\gamma_\omega}^{(h)},\\
\int_{\partial \Omega} g \Lambda(\gamma_\omega) h\dx[s]
%=\int_{\partial \Omega} g u_{\gamma_\omega}^{(h)}|_{\partial \Omega} \dx[s]
= \int_\Omega \gamma_\omega \nabla u_{\gamma_\omega}^{(g)} \cdot \nabla u_{\gamma_\omega}^{(h)}
=\int_{\partial \Omega} h \Lambda(\gamma_\omega) g\dx[s].
\end{eqnarray*}
In that sense, $\Lambda(\gamma_\omega)$ is symmetric but generally (for complex $\gamma_\omega$) not self-adjoint.

In simple two-point conductivity measurement setups, there exists an obvious monotonicity relation. Given a larger conductivity we
will require less voltage to drive the same current. Remarkably, this monotonicity relation extends to the case of continuous boundary
measurements. For real-valued conductivity functions $\sigma_1,\sigma_2\in L^\infty_+(\Omega;\R)$ we have that, for all $g\in \Ld^2(\partial \Omega)$,
\begin{eqnarray}
\nonumber \int_{\Omega}  \frac{\sigma_2}{\sigma_1} (\sigma_1-\sigma_2) \left|\nabla u_{\sigma_2}^{(g)}\right|^2 \dx\\
\leq \int_{\partial \Omega} g \left( \Lambda(\sigma_2) - \Lambda(\sigma_1) \right) g \dx[s]
\labeq{real_mono}\leq \int_{\Omega} (\sigma_1-\sigma_2) \left|\nabla u_{\sigma_2}^{(g)}\right|^2 \dx,
\end{eqnarray}
where $u_{\sigma_2}^{(g)}$ solves the EIT equation \req{EIT} with conductivity $\sigma_2$ and boundary currents $g$. Hence,
\[
\sigma_1\leq \sigma_2 \quad \mbox{ implies that } \quad \Lambda(\sigma_1)\geq \Lambda(\sigma_2),
\]
so that an imaging domain with larger conductivity yields to smaller measurements in the sense of operator definiteness.
The monotonicity relation \req{real_mono} goes back to Ikehata, Kang, Seo, and Sheen \cite{kang1997inverse,ikehata1998size}.
It is the basis of many results on inclusion detection in EIT, cf.\ \cite{kirsch2005factorization,ide2007probing,harrach2009detecting,harrach2010exact,harrach2010factorization,harrach2013monotonicity,harrach2013recent}.

The following lemma extends the relation \req{real_mono} to complex-valued conductivities (see also \cite{kirsch2005factorization,harrach2009detecting,harrach2010factorization} for similar results).
\begin{lemma}\label{lemma:monotonicity}
Let $\gamma_1,\gamma_2\in L_+^\infty(\Omega;\R)+\im L^\infty(\Omega;\R)$, $g\in \Ld^2(\partial \Omega)$,
and  $u_{\gamma_1}^{(g)},u_{\gamma_2}^{(g)}\in \Hd^1(\Omega)$ be the corresponding solutions of \req{EIT}.
% \begin{equation}\labeq{mon_lemma_EIT}
% \nabla \cdot (\gamma_j \nabla u_{\gamma_j}^{(g)}) = 0\quad \mbox{ in } \Omega \quad \mbox{ and } \quad
% \gamma_j \partial _\nu u_{\gamma_j}^{(g)}|_{\partial \Omega}=g,\quad j=1,2.
% \end{equation}}
Then
\begin{eqnarray*}
\int_{\Omega} \left( \frac{\Re(\gamma_2)}{\Re(\gamma_1)} \Re(\gamma_1-\gamma_2) - \frac{\Im(\gamma_2)^2}{\Re(\gamma_1)}\right) \left|\nabla u_{\gamma_2}^{(g)}\right|^2 \dx\\
\leq  \int_{\partial \Omega} \overline{g} \, \Re\left[ \Lambda(\gamma_2) -  \Lambda(\gamma_1) \right] g \dx[s]
\leq \int_{\Omega} \left( \Re(\gamma_1-\gamma_2) + \frac{\Im(\gamma_1)^2}{\Re(\gamma_1)}\right) \left|\nabla u_{\gamma_2}^{(g)}\right|^2 \dx.
\end{eqnarray*}
\end{lemma}
The proof of lemma \ref{lemma:monotonicity} is postponed to the end of this section.

\subsection{Detecting inclusions in the continuous case}\label{subsect:detect_cont}

We assume that the imaging domain $\Omega$ consists of a homogeneous background medium with one or several conductivity anomalies
(inclusions) $D$. For simplicity, we will present our result for the case that the anomalies possess a constant admittivity
and that the conductivity $\sigma_\omega$ and the permittivity $\epsilon_\omega$ do not change with frequency. More precisely,
we assume that $D\subset \Omega$ is a closed set with connected complement and that
$\gamma_0$ and $\gamma_\omega$ are given by
\begin{eqnarray}
\labeq{gamma_0} \gamma_0(x)=\left\{ \begin{array}{l l} \gamma_0^{(\Omega)}=\sigma_\Omega & \mbox{ for } x\in \Omega\setminus D\\
\gamma_0^{(D)}=\sigma_D & \mbox{ for } x\in D \end{array}\right.\\
\labeq{gamma_omega} \gamma_\omega(x)=\left\{ \begin{array}{l l} \gamma_\omega^{(\Omega)}=
\sigma_\Omega+\im \omega \epsilon_\Omega & \mbox{ for } x\in \Omega\setminus D\\
\gamma_\omega^{(D)}=
\sigma_D+\im \omega \epsilon_D & \mbox{ for } x\in D \end{array}\right.
\end{eqnarray}
with real-valued constants $\sigma_\Omega,\sigma_D,\epsilon_\Omega,\epsilon_D>0$. We also assume that the anomaly
fulfills
\begin{equation}
\labeq{contrast}
\epsilon_D \sigma_\Omega -   \epsilon_\Omega \sigma_D\neq 0,
\end{equation}
which is the contrast condition required to detect inclusion in weighted fdEIT, cf.\ \cite[Remark~2.3]{harrach2009detecting}.
Our results can easily be extended to inclusions of spatially varying and frequency-dependent admittivities as long as the background conductivities are constant.

The ratio of the background conductivities is denoted by
\begin{equation}
\labeq{alpha}
\alpha:=\frac{\gamma_\omega^{(\Omega)}}{\gamma_0^{(\Omega)}} =1+\im \omega \frac{\epsilon_\Omega}{\sigma_\Omega}.
\end{equation}
Obviously, $\alpha \Lambda(\gamma_\omega)=\Lambda(\gamma_\omega / \alpha)$.

We show that the anomaly $D$ can be detected from comparing (ratio-weighted) EIT measurements at a non-zero frequency $\omega>0$ with ultrasound-modulated DC measurements, i.e. that we can detect $D$ from knowledge of $\Lambda(\gamma_\omega)$,
$\Lambda(\gamma_0 (1+\beta \chi_B))$, and the background ratio $\alpha$. (Note that, the background ratio $\alpha$ could also be estimated by additionally taking unmodulated DC measurements $\Lambda(\gamma_0)$ and comparing them with $\Lambda(\gamma_\omega)$ in the same way as in weighted fdEIT, cf.\ \cite{seo2008frequency,harrach2009detecting,harrach2010factorization}.)

\begin{thm}\label{thm:main_cont}
Let $c:=\epsilon_D \sigma_\Omega -\epsilon_\Omega\sigma_D\neq 0$. 
\begin{enumerate}[(a)]
\item If $c>0$, then for sufficiently small $\beta>0$ and every open set $B\subseteq \Omega$, 
\begin{equation}
 B\subseteq D \quad \mbox{ if and only if } \quad
\Re\left(\alpha \Lambda( \gamma_\omega) \right)\leq \Lambda((1+\beta\chi_B)\gamma_0).
\end{equation}
\item If $c<0$, then for sufficiently small $\beta>0$ and every open set $B\subseteq \Omega$, 
\begin{equation}
 B\subseteq D \quad \mbox{ if and only if } \quad
\Re\left(\alpha \Lambda( \gamma_\omega) \right) \geq \Lambda((1-\beta\chi_B)\gamma_0).
\end{equation}
\end{enumerate}
The modulation strength $\beta>0$ is sufficiently small if
\[
\beta \leq \left\{ \begin{array}{l l} 
\omega^2 |c| \frac{\epsilon_\Omega}{\sigma_D(\sigma_\Omega^2 + \omega^2 \epsilon_\Omega^2)} & \mbox{ in case (a)},\\
\omega^2 |c|   \frac{\epsilon_D}{\sigma_D (\sigma_D \sigma_\Omega + \omega^2 \epsilon_D \epsilon_\Omega)}
& \mbox{ in case (b)}.
\end{array}\right.
\]
\end{thm}

Theorem~\ref{thm:main_cont} shows that, for sufficiently small modulation strengths,
the ultrasound-modulated DC measurements are larger ($c>0$), resp., smaller ($c<0$) than
(the self-adjoint part of ratio-weighted) measurements taken at a non-zero frequency if and only if the focusing region
lies inside the unknown inclusion $D$. The terms larger and smaller are to be understood in the sense of operator definiteness.

\begin{remark}\label{rem:stable}
The monotonicity tests in \ref{thm:main_cont} are stable in the following sense (cf.\ \cite[remark 3.5]{harrach2013monotonicity}). Let $A^\delta$ be a (w.l.o.g.\ self-adjoint) approximation to
the compact and self-adjoint operator
$A$, 
\[
\norm{ A^\delta - A}_{\mathcal{L}(\Ld^2(\partial \Omega))}<\delta,
\]
where $A:=\Lambda((1+\beta\chi_B)\gamma_0) - \Re\left(\alpha \Lambda( \gamma_\omega) \right)$ in
case (a) of theorem \ref{thm:main_cont} and 
$A:=\Re\left(\alpha \Lambda( \gamma_\omega) \right)-\Lambda((1-\beta\chi_B)\gamma_0)$ in case (b).

We consider the regularized definiteness test 
\begin{equation}\labeq{reg_def}
A^\delta\geq -\delta I.
\end{equation}
If $A\geq 0$, then $A^\delta\geq -\delta I$ will be fulfilled. On the other hand, if $A\not \geq 0$, then $A$ must possess a negative eigenvalue $\lambda<0$, so that $A^\delta\not\geq -\delta I$ for all $\delta<-\frac{\lambda}{2}$.

Hence, in order to determine whether a given focusing region lies inside the unknown inclusion,
it suffices to know the measurements up to a certain precision level $\delta>0$.
In that sense, also our arguably idealistic modeling of a perfectly focused ultrasound beam only has to be approximately valid.
\end{remark}

\subsection{Proof of lemma~\ref{lemma:monotonicity} and theorem~\ref{thm:main_cont}}

Our proof of theorem~\ref{thm:main_cont} relies on the monotonicity relation for complex conductivity EIT in
lemma~\ref{lemma:monotonicity} and the concept of localized potentials developed by one of the
authors in \cite{gebauer2008localized}. To prove lemma~\ref{lemma:monotonicity}, we will first show the
following auxiliary result that will also be useful for the case of electrode measurements.

\begin{lemma}\label{lemma:monotonicity_aux}
Let $\gamma_1,\gamma_2\in L_+^\infty(\Omega;\R)+\im L^\infty(\Omega;\R)$, $g\in \Ld^2(\partial \Omega)$,
and $u_1,u_2\in H^1(\Omega)$ fulfill
\begin{eqnarray*}
\int_\Omega \gamma_1 |\nabla u_1|^2 \dx= \int_\Omega \gamma_2 \nabla u_2 \cdot \overline{\nabla u_1} \dx,\\
\int_\Omega \gamma_2 |\nabla u_2|^2 \dx= \int_\Omega \gamma_1 \nabla u_1 \cdot \overline{\nabla u_2} \dx.
\end{eqnarray*}
Then
\begin{eqnarray*}
\int_{\Omega} \left( \frac{\Re(\gamma_2)}{\Re(\gamma_1)} \Re(\gamma_1-\gamma_2) - \frac{\Im(\gamma_2)^2}{\Re(\gamma_1)}\right) |\nabla u_2|^2 \dx\\
\leq  \int_\Omega \Re (\gamma_2) |\nabla u_2|^2 \dx - \int_\Omega \Re (\gamma_1) |\nabla u_1|^2 \dx\\
\leq \int_{\Omega} \left( \Re(\gamma_1-\gamma_2) + \frac{\Im(\gamma_1)^2}{\Re(\gamma_1)}\right) |\nabla u_2|^2 \dx.
\end{eqnarray*}
\end{lemma}
\begin{proof}
Since
\begin{eqnarray*}
\fl 0 \leq \int_{\Omega} \Re(\gamma_1)
\left| \nabla u_1 - \frac{\gamma_2}{\Re ( \gamma_1)} \nabla u_2 \right|^2\\
%
%\fl \quad = \int_{\Omega} \Re(\gamma_1) | \nabla u_1|^2 \dx
%- 2 \Re \left( \int_{\Omega} \gamma_2 \nabla u_2 \cdot \overline{\nabla u_1} \dx \right)
%+ \int_{\Omega} \frac{|\gamma_2|^2}{\Re ( \gamma_1)} |\nabla u_2|^2 \dx\\
%
\fl \quad = \Re\left( \int_{\Omega} \gamma_1 | \nabla u_1|^2 \dx
- 2  \int_{\Omega} \gamma_2 \nabla u_2 \cdot \overline{\nabla u_1} \dx \right)
+ \int_{\Omega} \frac{|\gamma_2|^2}{\Re (\gamma_1)} |\nabla u_2|^2 \dx\\
\fl \quad = - \int_{\Omega} \Re(\gamma_1) | \nabla u_1|^2 \dx
+ \int_{\Omega} \frac{|\gamma_2|^2}{\Re (\gamma_1)} |\nabla u_2|^2 \dx\\
\fl \quad = \int_{\Omega} \Re(\gamma_2) | \nabla u_2|^2 \dx - \int_{\Omega} \Re(\gamma_1) | \nabla u_1|^2 \dx
+ \int_{\Omega} \left(\frac{|\gamma_2|^2}{\Re (\gamma_1)} - \Re(\gamma_2)\right) |\nabla u_2|^2 \dx,\\
\end{eqnarray*}
the first inequality follows from
\[
 \frac{|\gamma_2|^2}{\Re ( \gamma_1)} - \Re(\gamma_2)
= \frac{\Re(\gamma_2)^2+\Im (\gamma_2)^2}{\Re ( \gamma_1)} - \Re(\gamma_2)
= \frac{\Re(\gamma_2)}{\Re(\gamma_1)} \Re(\gamma_2-\gamma_1) + \frac{\Im(\gamma_2)^2}{\Re(\gamma_1)}.
\]
Likewise we obtain
\begin{eqnarray*}
\fl 0 \leq \int_{\Omega} \Re(\gamma_1)
\left| \nabla u_1 -\frac{\overline{\gamma_1}}{\Re ( \gamma_1)} \nabla u_2 \right|^2\\
\fl \quad = \int_{\Omega} \Re(\gamma_1) | \nabla u_1|^2 \dx
- 2 \Re \left( \int_{\Omega} \overline{\gamma_1} \nabla u_2 \cdot \overline{\nabla u_1} \dx \right)
+ \int_{\Omega} \frac{|\gamma_1|^2}{\Re ( \gamma_1)} |\nabla u_2|^2 \dx\\
%
%\fl \quad = \int_{\Omega} \Re(\gamma_1) | \nabla u_1|^2 \dx
%- 2 \int_{\Omega} \Re(\gamma_2) | \nabla u_2|^2 \dx
%+ \int_{\Omega} \frac{|\gamma_1|^2}{\Re ( \gamma_1)} |\nabla u_2|^2 \dx\\
%
\fl \quad = \int_{\Omega} \Re(\gamma_1) | \nabla u_1|^2 \dx
- \int_{\Omega} \Re(\gamma_2) | \nabla u_2|^2 \dx
+ \int_{\Omega} \left( \frac{|\gamma_1|^2}{\Re ( \gamma_1)} - \Re(\gamma_2)\right) |\nabla u_2|^2 \dx,
\end{eqnarray*}
so that the second inequality follows from
\[
\frac{|\gamma_1|^2}{\Re ( \gamma_1)} - \Re(\gamma_2)
= \frac{\Re(\gamma_1)^2+\Im(\gamma_1)^2}{\Re ( \gamma_1)} - \Re(\gamma_2)
= \Re(\gamma_1-\gamma_2) + \frac{\Im(\gamma_1)^2}{\Re ( \gamma_1)}.
\]
\end{proof}

%%%%%%%%%%%%%%%%%%%%%%%%%%%%%%%%%%%%%%%%%%%%%%%%%%%%%

We also require the following elementary computation:
\begin{lemma}\label{lemma:tedious_computation}
Let $\gamma_0,\gamma_\omega:\ \Omega\to \C$, and $\alpha\in \C$ be given
by \req{gamma_0},\req{gamma_omega}, and \req{alpha}. Then, for all $\tilde\beta\in\mathbb{R}$, 
\begin{eqnarray*}
 \frac{\Re(\gamma_0)}{\Re(\gamma_\omega/\alpha)} \Re(\gamma_\omega/\alpha-\gamma_0) = \left\{ \begin{array}{l l} 0 & \mbox{ in } \Omega\setminus D,\\
  \frac{\epsilon_\Omega \sigma_D}{\sigma_\Omega} C
 & \mbox{ in } D,
\end{array}\right.\\
\Re(\gamma_\omega/\alpha-\gamma_0) + \frac{\Im(\gamma_\omega/\alpha)^2}{\Re(\gamma_\omega/\alpha)} = \left\{ \begin{array}{l l} 0 & \mbox{ in } \Omega\setminus D,\\
 \epsilon_D C
 & \mbox{ in } D,
\end{array}\right.\\
\Re(\gamma_\omega/\alpha-(1+\tilde\beta\chi_B)\gamma_0) = \left\{ \begin{array}{l l} -\tilde\beta\sigma_\Omega\chi_B & \mbox{ in } \Omega\setminus D,\\
 \sigma_D\left(\frac{\epsilon_\Omega}{\sigma_\Omega}C'-\tilde\beta\chi_B\right)
 & \mbox{ in } D,
\end{array}\right.\\
\Re(\gamma_\omega/\alpha-(1+\tilde\beta\chi_B)\gamma_0) + \frac{\Im(\gamma_\omega/\alpha)^2}{\Re(\gamma_\omega/\alpha)} = \left\{ \begin{array}{l l} -\tilde\beta\sigma_\Omega\chi_B & \mbox{ in } \Omega\setminus D,\\
 \epsilon_D C-\tilde\beta\sigma_D\chi_B
 & \mbox{ in } D,
\end{array}\right.
\end{eqnarray*}
where
\begin{equation}\labeq{Def_C}
C:=\omega^2 \frac{\epsilon_D \sigma_\Omega -\epsilon_\Omega\sigma_D}{\sigma_D \sigma_\Omega + \omega^2 \epsilon_D \epsilon_\Omega}\quad \mbox{and}\quad C':=\omega^2 \frac{\sigma_\Omega}{\sigma_D}\cdot\frac{
\epsilon_D \sigma_\Omega -\epsilon_\Omega\sigma_D}{\sigma_\Omega^2 + \omega^2 \epsilon_\Omega^2}.
\end{equation}
\end{lemma}
\begin{proof}
Let
\[
%\fl
\gamma_0=\left\{ \begin{array}{l l}
\gamma_0^{(\Omega)}=\sigma_\Omega & \mbox{ in } \Omega\setminus D,\\
\gamma_0^{(D)}=\sigma_D & \mbox{ in } D, \end{array}\right.
\quad
\gamma_\omega=\left\{ \begin{array}{l l}
\gamma_\omega^{(\Omega)}=\sigma_\Omega+\im \omega \epsilon_\Omega & \mbox{ in } \Omega\setminus D,\\
\gamma_\omega^{(D)}=\sigma_D+\im \omega \epsilon_D & \mbox{ in } D, \end{array}\right.
\]
with real-valued constants $\sigma_\Omega,\sigma_D,\epsilon_\Omega,\epsilon_D>0$,
and let $\alpha:=\gamma_\omega^{(\Omega)} / \gamma_0^{(\Omega)} = 1+\im \omega \frac{\epsilon_\Omega}{\sigma_\Omega}\in \C$.

Then, by definition of $\alpha$,
\[
\gamma_\omega/\alpha-\gamma_0 = 0 \quad \mbox{ in } \Omega\setminus D,\quad
\mbox{ and } \quad
\Im(\gamma_\omega/\alpha)=0 \quad \mbox{ in } \Omega\setminus D,
\]
so that
\begin{eqnarray*}
\frac{\Re(\gamma_0)}{\Re(\gamma_\omega/\alpha)} \Re(\gamma_\omega/\alpha-\gamma_0)
= 0\quad \mbox{ in } \Omega\setminus D,\\
\Re(\gamma_\omega/\alpha-\gamma_0) + \frac{\Im(\gamma_\omega/\alpha)^2}{\Re(\gamma_\omega/\alpha)}=0
 \quad \mbox{ in } \Omega\setminus D,\\
\Re(\gamma_\omega/\alpha-(1+\tilde\beta\chi_B)\gamma_0)=-\tilde\beta\sigma_\Omega\chi_B
 \quad \mbox{ in } \Omega\setminus D,\\
 \Re(\gamma_\omega/\alpha-(1+\tilde\beta\chi_B)\gamma_0) + \frac{\Im(\gamma_\omega/\alpha)^2}{\Re(\gamma_\omega/\alpha)}=-\tilde\beta\sigma_\Omega\chi_B
 \quad \mbox{ in } \Omega\setminus D.
\end{eqnarray*}

In $D$, we have that
\begin{eqnarray*}
\Re(\gamma_0) = \sigma_D,\\
\Re(\gamma_\omega / \alpha)=\Re \left(\gamma^{(D)}_\omega \frac{\gamma_0^{(\Omega)}}{\gamma_\omega^{(\Omega)}} \right)
= \sigma_\Omega \Re\left( \frac{\sigma_D+\im \omega \epsilon_D}{\sigma_\Omega+\im \omega \epsilon_\Omega}\right)
%= \sigma_\Omega \Re\left( \frac{(\sigma_D+\im \omega \epsilon_D) (\sigma_\Omega-\im \omega \epsilon_\Omega )}{\sigma_\Omega^2+\omega^2 \epsilon_\Omega^2}\right)
= \sigma_\Omega  \frac{\sigma_D \sigma_\Omega + \omega^2 \epsilon_D \epsilon_\Omega}{\sigma_\Omega^2+\omega^2 \epsilon_\Omega^2},\\
\Im(\gamma_\omega / \alpha)=\Im \left(\gamma^{(D)}_\omega \frac{\gamma_0^{(\Omega)}}{\gamma_\omega^{(\Omega)}} \right)
=\sigma_\Omega \Im\left( \frac{\sigma_D+\im \omega \epsilon_D}{\sigma_\Omega+\im \omega \epsilon_\Omega}\right)
%= \sigma_\Omega \Im\left( \frac{(\sigma_D+\im \omega \epsilon_D) (\sigma_\Omega-\im \omega \epsilon_\Omega )}{\sigma_\Omega^2+\omega^2 \epsilon_\Omega^2}\right)
= \omega \sigma_\Omega  \frac{\epsilon_D \sigma_\Omega -  \epsilon_\Omega \sigma_D}{\sigma_\Omega^2+\omega^2 \epsilon_\Omega^2}.\\
\end{eqnarray*}
Hence, in $D$,
\[
\fl \Re(\gamma_\omega/\alpha-\gamma_0)
 = \sigma_\Omega  \frac{\sigma_D \sigma_\Omega + \omega^2 \epsilon_D \epsilon_\Omega}{\sigma_\Omega^2+\omega^2 \epsilon_\Omega^2} - \sigma_D
 = \omega^2 \epsilon_\Omega \frac{ \epsilon_D \sigma_\Omega - \sigma_D \epsilon_\Omega}{\sigma_\Omega^2+\omega^2 \epsilon_\Omega^2}=\frac{\epsilon_\Omega \sigma_D}{\sigma_\Omega}C',
\]
which shows that
\[
\frac{\Re(\gamma_0)}{\Re(\gamma_\omega/\alpha)}\Re(\gamma_\omega/\alpha-\gamma_0)
%= \omega^2 \frac{\sigma_D \epsilon_\Omega}{\sigma_\Omega} \frac{\sigma_\Omega \epsilon_D  - \sigma_D \epsilon_\Omega}{\sigma_D \sigma_\Omega + \omega^2\epsilon_D\epsilon_\Omega}
= \omega^2 \sigma_D \epsilon_\Omega \frac{\epsilon_D \sigma_\Omega  - \epsilon_\Omega \sigma_D }{\sigma_\Omega(\sigma_D \sigma_\Omega + \omega^2\epsilon_D\epsilon_\Omega)}
= \frac{\epsilon_\Omega \sigma_D}{\sigma_\Omega} C,
\]
and
\[
\Re(\gamma_\omega/\alpha-(1+\tilde\beta\chi_B)\gamma_0)
= \sigma_D\left(\frac{\epsilon_\Omega}{\sigma_\Omega}C'-\tilde\beta\chi_B\right).
\]

The remaining two equalities follow from
\begin{eqnarray*}
\Re(\gamma_\omega/\alpha-\gamma_0) + \frac{\Im(\gamma_\omega/\alpha)^2}{\Re(\gamma_\omega/\alpha)}\\
= \omega^2 \epsilon_\Omega \frac{\epsilon_D \sigma_\Omega   - \epsilon_\Omega\sigma_D }{\sigma_\Omega^2+\omega^2 \epsilon_\Omega^2}
+ \omega^2 \sigma_\Omega  \frac{(\epsilon_D \sigma_\Omega -  \epsilon_\Omega \sigma_D)^2}{(\sigma_\Omega^2+\omega^2 \epsilon_\Omega^2)(\sigma_D \sigma_\Omega + \omega^2 \epsilon_D \epsilon_\Omega)}\\
= \omega^2 \frac{\epsilon_D \sigma_\Omega   - \epsilon_\Omega\sigma_D }{\sigma_\Omega^2+\omega^2 \epsilon_\Omega^2}
\left( \epsilon_\Omega +\sigma_\Omega \frac{\epsilon_D \sigma_\Omega   - \epsilon_\Omega \sigma_D }{\sigma_D \sigma_\Omega + \omega^2 \epsilon_D \epsilon_\Omega}\right)\\
= \omega^2 \frac{\epsilon_D \sigma_\Omega   - \epsilon_\Omega\sigma_D }{\sigma_\Omega^2+\omega^2 \epsilon_\Omega^2}\,
\frac{\omega^2 \epsilon_D\epsilon_\Omega^2 +\epsilon_D \sigma_\Omega^2}
{\sigma_D \sigma_\Omega + \omega^2 \epsilon_D \epsilon_\Omega}
= \omega^2 \epsilon_D \frac{\epsilon_D \sigma_\Omega   - \epsilon_\Omega\sigma_D }
{\sigma_D \sigma_\Omega + \omega^2 \epsilon_D \epsilon_\Omega}
=\epsilon_D C,
\end{eqnarray*}
which also yields
\[
 \Re(\gamma_\omega/\alpha-(1+\tilde\beta\chi_B)\gamma_0) + \frac{\Im(\gamma_\omega/\alpha)^2}{\Re(\gamma_\omega/\alpha)}
= \epsilon_D C - \tilde\beta\sigma_D\chi_B.
\]
\end{proof}
Now we are ready to prove lemma~\ref{lemma:monotonicity} and theorem~\ref{thm:main_cont}.

\paragraph{Proof of lemma \ref{lemma:monotonicity}.}
For all $g\in \Ld^2(\partial \Omega)$ we have that
\begin{eqnarray*}
\fl
\int_{\partial \Omega} \overline{g} \Lambda(\gamma_1) g \dx[s]&= \int_{\partial \Omega}  \overline{g} u_{\gamma_1}^{(g)}|_{\partial \Omega} \dx[s]= \int_\Omega \overline{\gamma_1} \left|\nabla u_{\gamma_1}^{(g)}\right|^2 \dx= \int_\Omega \overline{\gamma_2 \nabla u_{\gamma_2}^{(g)}} \cdot \nabla u_{\gamma_1}^{(g)} \dx,\\
\fl
\int_{\partial \Omega} \overline{g} \Lambda(\gamma_2) g \dx[s]&= \int_{\partial \Omega}  \overline{g} u_{\gamma_2}^{(g)}|_{\partial \Omega} \dx[s]= \int_\Omega \overline{\gamma_2} \left|\nabla u_{\gamma_2}^{(g)}\right|^2 \dx= \int_\Omega \overline{\gamma_1 \nabla u_{\gamma_1}^{(g)}} \cdot \nabla u_{\gamma_2}^{(g)} \dx,
\end{eqnarray*}
so that the assertion of lemma \ref{lemma:monotonicity} immediately follows from lemma \ref{lemma:monotonicity_aux}.  \hfill $\Box$

\paragraph{Proof of theorem~\ref{thm:main_cont}.}
\begin{enumerate}[(a)]
\item 
\begin{enumerate}[(i)]
\item Let $c:=\epsilon_D \sigma_\Omega -\epsilon_\Omega\sigma_D>0$, and $B\subseteq D$.

We use the first inequality in lemma \ref{lemma:monotonicity} with $\gamma_2:=(1+\beta\chi_B)\gamma_0$, and $\gamma_1:= \gamma_\omega/\alpha$ 
together with the third equality in lemma~\ref{lemma:tedious_computation} with $\tilde \beta:=\beta$
to obtain that, for all $g\in \Ld^2(\partial \Omega)$,
\begin{eqnarray*}
\int_{\partial \Omega} \overline{g} \left[ \Lambda((1+\beta\chi_B)\gamma_0) - \Re\left( \Lambda(\gamma_\omega/\alpha)\right) \right] g \dx[s]\\
\geq \int_{\Omega}  \frac{\Re((1+\beta\chi_B)\gamma_0)}{\Re(\gamma_\omega/\alpha)} \Re(\gamma_\omega/\alpha-(1+\beta\chi_B)\gamma_0)  \left|\nabla u_{\gamma_2}^{(g)}\right|^2 \dx\\
= \int_{D} \frac{(1+\beta\chi_B)\gamma_0}{\Re(\gamma_\omega/\alpha)}\sigma_D\left(\frac{\epsilon_\Omega}{\sigma_\Omega}C'-\beta\chi_B\right)\left|\nabla u_{\gamma_2}^{(g)}\right|^2 \dx,
\end{eqnarray*}
where $C'$ is defined by \req{Def_C} in lemma~\ref{lemma:tedious_computation}. The right hand side is non-negative if 
\[
\beta\leq \frac{\epsilon_\Omega}{\sigma_\Omega}C' = \omega^2 |c| \frac{\epsilon_\Omega}{\sigma_D(\sigma_\Omega^2 + \omega^2 \epsilon_\Omega^2)},
\]
so that, for sufficiently small $\beta>0$,
\begin{equation*}
 B\subseteq D \quad \mbox{ implies } \quad
\Re\left(\alpha \Lambda( \gamma_\omega) \right)\leq \Lambda((1+\beta\chi_B)\gamma_0).
\end{equation*}
\item Now let $c:=\epsilon_D \sigma_\Omega -\epsilon_\Omega\sigma_D>0$, and $B\not\subseteq D$.

We use the second inequality in lemma \ref{lemma:monotonicity} with $\gamma_2:=\gamma_0$, and $\gamma_1:= \gamma_\omega/\alpha$ together with the second equality in lemma~\ref{lemma:tedious_computation} to obtain that, for all $g\in \Ld^2(\partial \Omega)$,
\begin{eqnarray*}
\nonumber \int_{\partial \Omega} \overline{g} \left[ \Lambda(\gamma_0) - \Re\left( \Lambda(\gamma_\omega/\alpha)\right) \right] g \dx[s]\\
\labeq{main_cont_hilf1}
\leq \int_{\Omega} \left( \Re(\gamma_\omega/\alpha-\gamma_0) + \frac{\Im(\gamma_\omega/\alpha)^2}{\Re(\gamma_\omega/\alpha)}\right) \left|\nabla u_{\gamma_0}^{(g)}\right|^2 \dx\\
\nonumber = \epsilon_D C \int_{D} \left|\nabla u_{\gamma_0}^{(g)}\right|^2 \dx,
\end{eqnarray*}
where $C$ is defined by \req{Def_C} in lemma~\ref{lemma:tedious_computation}. The first inequality in lemma \ref{lemma:monotonicity} with $\gamma_2:=\gamma_0$ and $\gamma_1:=(1+\beta\chi_B)\gamma_0$ yields that, for all $g\in \Ld^2(\partial \Omega)$,
\begin{equation*}
\labeq{main_cont_hilf2}
\int_B  \frac{\beta}{1+\beta} \gamma_0 \left|\nabla u_{\gamma_0}^{(g)}\right|^2 \dx
\leq \int_{\partial \Omega} \overline{g} \left[ \Lambda(\gamma_0) - \Lambda((1+\beta\chi_B)\gamma_0)\right] g \dx[s].
\end{equation*}
Combining both inequalities, we obtain that, for all $g\in \Ld^2(\partial \Omega)$,
\begin{eqnarray*}
\int_{\partial \Omega} \overline{g} \left[ \Lambda((1+\beta\chi_B)\gamma_0) - \Re\left( \alpha\Lambda(\gamma_\omega)\right)\right] g \dx[s]\\
\leq \epsilon_D C \int_{D} \left|\nabla u_{\gamma_0}^{(g)}\right|^2 \dx -\int_B  \frac{\beta}{1+\beta} \gamma_0 \left|\nabla u_{\gamma_0}^{(g)}\right|^2 \dx.
\end{eqnarray*}

Now we apply the technique of localized potentials \cite{gebauer2008localized,harrach2013monotonicity} to show that the right hand side of this inequality attains negative values. Since $B\not\subseteq D$ we can choose a smaller open subset $B'\subseteq B$ with $\overline {B'}\cap D=\emptyset$. Since $D\subset \Omega$ and $\Omega\setminus D$ is connected,
we obtain from \cite[Thm.~3.6]{harrach2013monotonicity} a sequence of currents $(g_k)_{k\in \N}\subset \Ld^2(\partial \Omega)$, so that the solutions $(u^{(g_k)})_{k\in \N}\subset \Hd^1(\Omega)$
of
\[
\Delta u^{(g_k)}=0, \quad \partial_\nu u^{(g_k)}|_{\partial \Omega}=g_k
\]
fulfill
\[
\lim_{k\to \infty} \int_{B'} |\nabla u^{(g_k)}|^2\dx= \infty \quad \mbox{ and } \quad \lim_{k\to \infty} \int_{D} |\nabla u^{(g_k)}|^2\dx=0.
\]
Since $\gamma_0$ is constant on $\Omega\setminus D$, \cite[Lemma~3.7]{harrach2013monotonicity} yields that also the corresponding solutions
$(u_{\gamma_0}^{(g_k)})_{k\in \N}\subset \Hd^1(\Omega)$ of \req{EIT}
fulfill
\[
\lim_{k\to \infty} \int_{B'} |\nabla u_{\gamma_0}^{(g_k)}|^2\dx= \infty \quad \mbox{ and } \quad \lim_{k\to \infty} \int_{D} |\nabla u_{\gamma_0}^{(g_k)}|^2\dx=0.
\]
Hence, with this sequence of currents,
\[
\int_{\partial \Omega} \overline{g_k} \left[ \Lambda((1+\beta\chi_B)\gamma_0) - \Re\left( \alpha\Lambda(\gamma_\omega)\right)\right] g_k \dx[s]\to -\infty,
\]
which shows that, for all $\beta>0$,
\begin{equation*}
 B\not\subseteq D \quad \mbox{ implies } \quad
\Re\left(\alpha \Lambda( \gamma_\omega) \right)\not\leq \Lambda((1+\beta\chi_B)\gamma_0).
\end{equation*}
\end{enumerate}
\item 
\begin{enumerate}[(i)]
\item Let $c:=\epsilon_D \sigma_\Omega -\epsilon_\Omega\sigma_D<0$, and $B\subseteq D$.

We use the second inequality in lemma \ref{lemma:monotonicity} with $\gamma_2:=(1-\beta\chi_B)\gamma_0$, and $\gamma_1:= \gamma_\omega/\alpha$ 
together with the fourth equality in lemma~\ref{lemma:tedious_computation} with $\tilde \beta:=-\beta$
to obtain that, for all $g\in \Ld^2(\partial \Omega)$,
\begin{eqnarray*}
\int_{\partial \Omega} \overline{g} \left[ \Lambda((1- \beta\chi_B)\gamma_0) - \Re\left( \alpha\Lambda(\gamma_\omega)\right)\right] g \dx[s]\\
\leq \int_{\Omega} \left( \Re(\gamma_\omega/\alpha-(1-\beta\chi_B)\gamma_0) + \frac{\Im(\gamma_\omega/\alpha)^2}{\Re(\gamma_\omega/\alpha)}\right) \left|\nabla u_{\gamma_2}^{(g)}\right|^2 \dx.\\
=\int_{D} (\epsilon_D C +\beta\sigma_D\chi_B)\left|\nabla u_{\gamma_2}^{(g)}\right|^2 \dx.
\end{eqnarray*}
The right hand side is non-positive if 
\[
\beta\leq -\frac{\epsilon_D}{\sigma_D}C = \omega^2 |c|   \frac{\epsilon_D}{\sigma_D (\sigma_D \sigma_\Omega + \omega^2 \epsilon_D \epsilon_\Omega)},
\]
so that, for sufficiently small $\beta>0$,
\begin{equation*}
 B\subseteq D \quad \mbox{ implies } \quad
\Re\left(\alpha \Lambda( \gamma_\omega) \right)\geq \Lambda((1-\beta\chi_B)\gamma_0).
\end{equation*}
%%%%%%%%%%
\item Now let $c:=\epsilon_D \sigma_\Omega -\epsilon_\Omega\sigma_D<0$, and $B\not\subseteq D$.

We use the first inequality in lemma \ref{lemma:monotonicity} with $\gamma_2:=\gamma_0$, and $\gamma_1:= \gamma_\omega/\alpha$ together with the first equality in lemma~\ref{lemma:tedious_computation} to obtain that, for all $g\in \Ld^2(\partial \Omega)$,
\begin{eqnarray*}
\nonumber \int_{\partial \Omega} \overline{g} \left[ \Lambda(\gamma_0) - \Re\left( \Lambda(\gamma_\omega/\alpha)\right) \right] g \dx[s]\\
\geq \int_{\Omega}  \frac{\Re(\gamma_0)}{\Re(\gamma_\omega/\alpha)} \Re(\gamma_\omega/\alpha-\gamma_0)  \left|\nabla u_{\gamma_0}^{(g)}\right|^2 \dx\\
= \frac{\epsilon_\Omega\sigma_D}{\sigma_\Omega} C \int_{D} \left|\nabla u_{\gamma_0}^{(g)}\right|^2 \dx.
\end{eqnarray*}
The second inequality in lemma \ref{lemma:monotonicity} with $\gamma_2:=\gamma_0$ and $\gamma_1:=(1-\beta\chi_B)\gamma_0$ yields that, for all $g\in \Ld^2(\partial \Omega)$,
\begin{equation*}
\int_{\partial \Omega} \overline{g} \left[ \Lambda(\gamma_0) - \Lambda((1-\beta\chi_B)\gamma_0)\right] g \dx[s]\\
\leq - \int_B  \beta\gamma_0 \left|\nabla u_{\gamma_0}^{(g)}\right|^2 \dx.
\end{equation*}
Combining both inequalities, we obtain that, for all $g\in \Ld^2(\partial \Omega)$,
\begin{eqnarray*}
\int_{\partial \Omega} \overline{g} \left[ \Lambda((1-\beta\chi_B)\gamma_0) - \Re\left( \alpha\Lambda(\gamma_\omega)\right)\right] g \dx[s]\\
\geq \frac{\epsilon_\Omega\sigma_D}{\sigma_\Omega} C \int_{D} \left|\nabla u_{\gamma_0}^{(g)}\right|^2 \dx +
\int_B  \beta\gamma_0 \left|\nabla u_{\gamma_0}^{(g)}\right|^2 \dx.
\end{eqnarray*}
The same localized potentials argument as in part (a)(ii) shows that there exists a sequence of currents such that 
\[
\int_{\partial \Omega} \overline{g} \left[ \Lambda((1-\beta\chi_B)\gamma_0) - \Re\left( \alpha\Lambda(\gamma_\omega)\right)\right] g \dx[s]\to \infty.
\]
Hence, for all $\beta>0$,
\begin{equation*}
 B\not\subseteq D \quad \mbox{ implies } \quad
\Re\left(\alpha \Lambda( \gamma_\omega) \right)\not\geq \Lambda((1-\beta\chi_B)\gamma_0).
\end{equation*}
\end{enumerate}
\end{enumerate}
\hfill $\Box$

\section{Electrode measurements}
\label{Sec:electrodes}

\subsection{The setting}
\label{Subsec:setting_electrodes}

In a realistic setting, the currents will be applied using a finite number of electrodes $\mathcal E_l\subset \partial \Omega$, $l=1,\ldots,m$,
that are attached to the imaging domain's surface. We assume that the electrodes are perfectly conducting and that contact impedances are negligible (the so-called \emph{shunt model}, cf., e.g., \cite{Che99}). Driving a current $I_l\in \C$ through the $l$-th electrode, with $\sum_{l=1}^m I_l=0$, the electric potential is given by the solution
$u_{\gamma_\omega}\in H^1_{\mathcal E}(\Omega)$ of
\begin{eqnarray}
\labeq{shunt1} \nabla \cdot (\gamma_\omega \nabla u_{\gamma_\omega}) = 0\quad  \mbox{ in } \Omega,\\
\labeq{shunt2} \int_{\mathcal E_l} \gamma_\omega \partial_\nu u_{\gamma_\omega} \dx[s] = I_l \quad \mbox{ for } l=1,\ldots,m,\\
\labeq{shunt3} \gamma_\omega \partial_\nu u_{\gamma_\omega} =0 \quad  \mbox{ on } \partial \Omega\setminus \bigcup_{l=1}^m \mathcal E_l,\\
\labeq{shunt4} u_{\gamma_\omega}|_{\mathcal E_l}=\mathrm{const.} \quad \forall j=1,\ldots,m,
\end{eqnarray}
where $H^1_{\mathcal E}(\Omega)$ is the subspace of $H^1$-functions that are locally constant on each $\mathcal E_l$, $l=1,\ldots,m$ and these constants sum up to zero.

We assume that the voltage-current-measurements are carried out in the following complete dipole-dipole configuration.
Let $(j_r,k_r)$, $r=1,\ldots,N$ be a set of electrode pairs with $j_r\neq k_r$. For each of these pairs, $r=1,\ldots,N$, a current of $I_{j_r}=1$ and $I_{k_r}=-1$ is driven through the $j_r$-th and the $k_r$-th electrode, respectively. The other electrodes are kept insulated. The resulting electric potential inside the imaging domain is given by the solution $u^{\langle r\rangle}_{\gamma_\omega}\in H^1_{\mathcal E}(\Omega)$ of \req{shunt1}--\req{shunt4} with $I_l=\delta_{l,j_r}-\delta_{l,k_r}$, $l=1,\ldots,N$.

While the current is driven through the $r$-th pair of electrodes, we measure the required voltage difference
on all pairs of electrodes, i.e., between the $j_s$ and the $k_s$ electrode for all $s=1,\ldots,N$. We collect these measurements in the matrix
\[
R(\gamma_\omega) = \left( u^{\langle r\rangle}_{\gamma_\omega}|_{\mathcal E_{j_s}}- u^{\langle r\rangle}_{\gamma_\omega}|_{\mathcal E_{k_s}} \right)_{r,s=1,\ldots,N}\in \C^{N\times N}.
\]

Let us comment on our use of the shunt electrode model. It seems to be widely accepted that the most
accurate electrode model in EIT is the \emph{complete electrode model}, cf., e.g., \cite{Che99}, where not only the shunting effects but also contact impedances between the electrodes and the imaging domain are taken into account. The effect of contact impedances is often neglected in the case that voltages are not measured on current driven electrodes, but our method requires such measurements, see below.
Contact impedances can also be neglected in the case of DC difference measurements on point electrodes,
see \cite{hanke2011justification}. Since both, the effect of an ultrasound modulation and the effect of a (weighted) frequency change on the measurements are widely analogous to using DC difference measurements, we believe that our use of the shunt model is justified for sufficiently small electrodes, though this has yet to be justified rigorously.

We also stress that our method relies on the matrix structure of the measurements $R$, which means that the same electrode pairs have to be used for measuring voltages and applying currents. In particular, we require voltage measurements on current driven electrodes (for the three main diagonals in $R$). The simultaneous measurement of voltage and current is usually considered problematic and these measurements are avoided in traditional EIT approaches. Nevertheless, successful reconstructions have already been obtained in practical phantom experiments with methods requiring the full matrix such as the factorization method and monotonicity-based methods, cf. \cite{harrach2010factorization,zhou_preprint}. Also, the recent preprint \cite{harrachinterpolation} studies the possibility of interpolating the voltages on current-driven electrodes from the measurements on current-free electrodes.

\subsection{Monotonicity results for the shunt model}\label{Subsec:Monotonicity_results_electrode}

As in the continuous case, we will compare measurements in the sense of matrix definiteness.
We define the self-adjoint part of a matrix $A\in \C^{N\times N}$ by setting
\[
\Re(A):=\frac{1}{2}(A+A^*)
\]
where $A^*\in \C^{N\times N}$ is the adjoint (conjugate transpose) of $A$, i.e.,
\[
g^* (A h)  = (Ag)^* h \quad \mbox{ for all } g,h\in \C^N, \quad \mbox{ and } \quad g^*=\overline{g}^T.
\]
Obviously, $\Re(A)$ is self-adjoint.

For two self-adjoint matrices $A,B\in \C^{N\times N}$, we write $A\leq B$ if
$B-A$ is positive semidefinite, i.e.
\[
g^* A g\leq g^* B g \quad \forall g\in \C^N.
\]
This is equivalent to the fact that all eigenvalues of $B-A$ are non-negative.

Note that the entries of the measurement matrix $R(\gamma_\omega)$ satisfy
\begin{eqnarray*}
 u^{\langle r\rangle}_{\gamma_\omega}|_{\mathcal E_{j_s}}- u^{\langle r\rangle}_{\gamma_\omega}|_{\mathcal E_{k_s}}
 = u^{\langle r\rangle}_{\gamma_\omega}|_{\mathcal E_{j_s}} \int_{\mathcal E_{j_s}} \gamma_\omega \partial_\nu u^{\langle s\rangle}_{\gamma_\omega} \dx[s]
 + u^{\langle r\rangle}_{\gamma_\omega}|_{\mathcal E_{k_s}} \int_{\mathcal E_{k_s}} \gamma_\omega \partial_\nu u^{\langle s\rangle}_{\gamma_\omega} \dx[s]\\
 = \int_{\partial \Omega}\gamma_\omega \partial_\nu u^{\langle s\rangle}_{\gamma_\omega}|_{\partial \Omega} \, u^{\langle r\rangle}_{\gamma_\omega}|_{\partial \Omega}
 \dx[s] = \int_\Omega \gamma_\omega \nabla u_\omega^{(r)}\cdot \nabla u^{\langle s\rangle}_{\gamma_\omega}
 = u^{\langle s\rangle}_{\gamma_\omega}|_{\mathcal E_{j_r}}- u^{\langle s\rangle}_{\gamma_\omega}|_{\mathcal E_{k_r}}.
\end{eqnarray*}
Hence $R(\gamma_\omega)$ is a symmetric, but generally (for complex $\gamma_\omega$) not self-adjoint matrix.
This also shows that the self-adjoint part of the measurement matrix $\Re(R(\gamma_\omega))$ is identical to the
matrix containing the real part of each voltage measurement
\[
\Re(R(\gamma_\omega)) = \left( \Re(u^{\langle r\rangle}_{\gamma_\omega})|_{\mathcal E_{j_s}}- \Re(u^{\langle r\rangle}_{\gamma_\omega})|_{\mathcal E_{k_s}} \right)_{r,s=1,\ldots,N}\in \R^{N\times N}.
\]

The montonicity estimate from the continuous case can be extended to the case of electrode measurements.
\begin{lemma}\label{lemma:monotonicity_shunt}
Let $\gamma_1,\gamma_2\in L_+^\infty(\Omega;\R)+\im L^\infty(\Omega;\R)$, $g=(g_r)_{r=1}^N\in \C^N$
and  $u^{[g]}_{\gamma_\tau}\in H^1_{\mathcal E}(\Omega)$ ($\tau=1,2$) denote the solution of
\begin{eqnarray*}
\nabla \cdot (\gamma_\tau \nabla u^{[g]}_{\gamma_\tau}) = 0\quad  \mbox{ in } \Omega,\\
\int_{\mathcal E_{l}} \gamma_\tau \partial_\nu u^{[g]}_{\gamma_\tau} \dx[s] = \sum_{r:\ j_r=l} g_r - \sum_{r:\ k_r=l} g_r
\quad \mbox{ for all } l=1,\ldots,m,\\
\gamma_\tau \partial_\nu u^{[g]}_{\gamma_\tau} =0 \quad  \mbox{ on } \partial \Omega\setminus \bigcup_{l=1}^m \mathcal E_l,\\
u^{[g]}_{\gamma_\tau}|_{\mathcal E_l}=\mathrm{const.} \quad \forall l=1,\ldots,m.
\end{eqnarray*}
Then
\begin{eqnarray*}
\int_{\Omega} \left( \frac{\Re(\gamma_2)}{\Re(\gamma_1)} \Re(\gamma_1-\gamma_2) - \frac{\Im(\gamma_2)^2}{\Re(\gamma_1)}\right) \left|\nabla u^{[g]}_{\gamma_2}\right|^2 \dx\\
\leq  g^* \,  \Re\left[ R(\gamma_2)-  R(\gamma_1)\right]  g
\leq \int_{\Omega} \left( \Re(\gamma_1-\gamma_2) + \frac{\Im(\gamma_1)^2}{\Re(\gamma_1)}\right) \left|\nabla u^{[g]}_{\gamma_2}\right|^2 \dx.
\end{eqnarray*}
\end{lemma}
The proof of lemma~\ref{lemma:monotonicity_shunt} is postponed to the end of this section.

\subsection{Detecting inclusions from electrode measurements}

We make the same assumptions as for the continuous case in subsection~\ref{subsect:detect_cont}.
The inclusion (or anomaly) $D\subset \Omega$ is assumed to be a closed set with connected complement.
$\gamma_0$ and $\gamma_\omega$ are assumed to be given by
\begin{eqnarray*}
\gamma_0(x)=\left\{ \begin{array}{l l} \gamma_0^{(\Omega)}=\sigma_\Omega & \mbox{ for } x\in \Omega\\
\gamma_0^{(D)}=\sigma_D & \mbox{ for } x\in D \end{array}\right.\\
\gamma_\omega(x)=\left\{ \begin{array}{l l} \gamma_\omega^{(\Omega)}=\sigma_\Omega+\im \omega \epsilon_\Omega & \mbox{ for } x\in \Omega\\
\gamma_\omega^{(D)}=\sigma_D+\im \omega \epsilon_D & \mbox{ for } x\in D \end{array}\right.
\end{eqnarray*}
with real-valued constants $\sigma_\Omega,\sigma_D,\epsilon_\Omega,\epsilon_D>0$. The anomaly is assumed to fulfill
the contrast condition \req{contrast}, i.e., $\epsilon_D \sigma_\Omega -   \epsilon_\Omega \sigma_D\neq 0$, and
\[
\alpha:=\frac{\gamma_\omega^{(\Omega)}}{\gamma_0^{(\Omega)}} =1+\im \omega \frac{\epsilon_\Omega}{\sigma_\Omega}.
\]
denotes the ratio of the background conductivities. Obviously, $\alpha R(\gamma_\omega)=R(\gamma_\omega / \alpha)$.

As in section \ref{Sec:continuous}, the results in this section can easily be extended to inclusions of spatially varying and frequency-dependent admittivities as long as the background conductivities are constant.

Our results for continuous boundary data suggest to compare, for sufficiently small modulation strengths $\beta>0$, the matrix of ultrasound-modulated DC measurements $R((1+\beta\chi_B)\gamma_0)$ with the
(self-adjoint part of the ratio-weighted) matrix of measurements taken at a non-zero frequency $R( \gamma_\omega)$.
This comparison (in the sense of matrix definiteness) should yield information about whether the
focusing region $B$ lies inside the unknown inclusion $D$. Indeed, we can prove the following theorem.

\begin{thm}\label{thm:electrode}
Let $c:=\epsilon_D \sigma_\Omega -\epsilon_\Omega\sigma_D\neq 0$. 
\begin{enumerate}[(a)]
\item If $c>0$, then for sufficiently small $\beta>0$ and every open set $B\subseteq \Omega$, 
\begin{equation}\label{eq:main_impl_elect_cpos}
 B\subseteq D \quad \mbox{ implies that } \quad
\Re\left(\alpha  R( \gamma_\omega) \right) \leq R((1+\beta\chi_B)\gamma_0).
\end{equation}
\item If $c<0$, then for sufficiently small $\beta>0$ and every open set $B\subseteq \Omega$, 
\begin{equation}\label{eq:main_impl_elect_cneg}
 B\subseteq D \quad \mbox{ implies that } \quad
\Re\left(\alpha  R( \gamma_\omega) \right) \geq R((1-\beta\chi_B)\gamma_0).
\end{equation}
\end{enumerate}
The modulation strength $\beta>0$ is sufficiently small if
\[
\beta \leq \left\{ \begin{array}{l l} 
\omega^2 |c| \frac{\epsilon_\Omega}{\sigma_D(\sigma_\Omega^2 + \omega^2 \epsilon_\Omega^2)} & \mbox{ in case (a)},\\
\omega^2 |c|   \frac{\epsilon_D}{\sigma_D (\sigma_D \sigma_\Omega + \omega^2 \epsilon_D \epsilon_\Omega)}
& \mbox{ in case (b)}.
\end{array}\right.
\]
\end{thm}

The converses of the implications \req{main_impl_elect_cpos} and \req{main_impl_elect_cneg} will generally not be true in the case of measurements with a finite number of electrodes.
However, when we increase the number of electrodes used for the measurements, then we can expect that the measurement matrices $R( \gamma_\omega)$ and $R((1+\beta\chi_B)\gamma_0)$ more and more resemble their continuous counterparts,
the Neumann-to-Dirichlet operators, cf.\ the works of Hakula, Hyv{\"o}nen and Lechleiter \cite{hyvonen2004complete,lechleiter2008factorization,hyvonen2009approximating}.
In fact, we can give the following intuitive justification of the converses of the implications in theorem \ref{thm:electrode} for sufficiently many electrodes in the spirit of \cite{harrach2010factorization}.

\begin{remark}\label{remark:electrode}
Let $B\not\subseteq D$ and $\beta>0$.
If there exists a current pattern $g=(g_r)_{r=1}^N\in \C^N$ such that the resulting DC potential
\[
u^{[g]}_{\gamma_0}:=\sum_{r=1}^N g_r u_{\gamma_0}^{\langle r \rangle}
\]
possesses a very large energy in $B\setminus D$ and a very small energy in $D$, then
\[
\Re\left(\alpha R( \gamma_\omega) \right) \not\leq R((1+\beta\chi_B)\gamma_0)\quad \mbox{if}\quad c>0
\]
or
\[
\Re\left(\alpha R( \gamma_\omega) \right) \not\geq R((1-\beta\chi_B)\gamma_0)\quad \mbox{if}\quad c<0.
\]
\end{remark}

\subsection{Proof of lemma~\ref{lemma:monotonicity_shunt}, theorem \ref{thm:electrode} and justification of remark~\ref{remark:electrode}}

\paragraph{Proof of lemma~\ref{lemma:monotonicity_shunt}.}
Let $g=(g_r)_{r=1}^N\in \C^N$. First note that for $\tau=1,2$, by linearity,
\[
u^{[g]}_{\gamma_\tau}=\sum_{r=1}^N u^{\langle r \rangle}_{\gamma_\tau} g_r \quad \mbox{ and } \quad
\sum_{r=1}^N  g_r\left(u^{\langle r \rangle}_{\gamma_\tau}|_{\mathcal E_{j_s}} -  u^{\langle r \rangle}_{\gamma_\tau}|_{\mathcal E_{k_s}}\right)
= u^{[g]}_{\gamma_\tau}|_{\mathcal E_{j_s}} -  u^{[g]}_{\gamma_\tau}|_{\mathcal E_{k_s}}.
\]
We thus obtain
\begin{eqnarray*}
g^* R(\gamma_1) g = \sum_{s=1}^N \overline{g_s} \left( u^{[g]}_{\gamma_1}|_{\mathcal E_{j_s}} -  u^{[g]}_{\gamma_1}|_{\mathcal E_{k_s}} \right)
= \sum_{l=1}^m \left( \sum_{s:\ j_s=l} \overline{g_s}  -  \sum_{s:\ k_s=l} \overline{g_s}\right) u^{[g]}_{\gamma_1}|_{\mathcal E_l}\\
= \sum_{l=1}^m \int_{\mathcal E_{l}} \overline{\gamma_1 \partial_\nu u^{[g]}_{\gamma_1}}|_{\mathcal E_l} \, u^{[g]}_{\gamma_1}|_{\mathcal E_l} \dx[s]
= \int_{\partial \Omega} \overline{\gamma_1 \partial_\nu u^{[g]}_{\gamma_1}} \, u^{[g]}_{\gamma_1} \dx[s]= \int_\Omega \overline{\gamma_1} \left|\nabla u^{[g]}_{\gamma_1}\right|^2\dx\\[0.5em]
%\mbox{or}\\[0.5em]
= \sum_{l=1}^m \int_{\mathcal E_{l}} \overline{\gamma_2 \partial_\nu u^{[g]}_{\gamma_2}}|_{\mathcal E_l} \, u^{[g]}_{\gamma_1}|_{\mathcal E_l} \dx[s]
= \int_{\partial \Omega} \overline{\gamma_2 \partial_\nu u^{[g]}_{\gamma_2}} \, u^{[g]}_{\gamma_1} \dx[s]= \int_\Omega \overline{\gamma_2 \nabla u^{[g]}_{\gamma_2}}\cdot \nabla u^{[g]}_{\gamma_1} \dx
\end{eqnarray*}
and likewise
\begin{eqnarray*}
g^* R(\gamma_2) g
= \int_\Omega \overline{\gamma_2} \left|\nabla u^{[g]}_{\gamma_2}\right|^2\dx = \int_\Omega \overline{\gamma_1 \nabla u^{[g]}_{\gamma_1}}\cdot \nabla u^{[g]}_{\gamma_2} \dx.
\end{eqnarray*}
Hence, the assertion follows from lemma \ref{lemma:monotonicity_aux}.\hfill $\Box$

\paragraph{Proof of theorem \ref{thm:electrode}.} The proof in identical to that of theorem \ref{thm:main_cont}(a)(i) and (b)(i) with lemma \ref{lemma:monotonicity_shunt} replacing lemma \ref{lemma:monotonicity}.\hfill $\Box$

\paragraph{Justification of remark~\ref{remark:electrode}.} As in theorem \ref{thm:main_cont}(a)(ii) and (b)(ii) (with lemma \ref{lemma:monotonicity_shunt} replacing lemma \ref{lemma:monotonicity}), we obtain that, for all $g\in \C^N$,
\begin{eqnarray*}
\fl g^* \left[ R((1+\beta\chi_B)\gamma_0) - \Re\left( \alpha R(\gamma_\omega)\right)\right] g 
\leq \epsilon_D C \int_{D} \left|\nabla u^{[g]}_{\gamma_0}\right|^2 \dx -\int_B  \frac{\beta}{1+\beta} \gamma_0 \left|\nabla u^{[g]}_{\gamma_0}\right|^2 \dx.
\end{eqnarray*}
and
\begin{eqnarray*}
\fl g^* \left[ R((1-\beta\chi_B)\gamma_0) - \Re\left( \alpha R(\gamma_\omega)\right)\right] g 
\geq \frac{\epsilon_\Omega\sigma_D}{\sigma_\Omega} C \int_{D} \left|\nabla u^{[g]}_{\gamma_0}\right|^2 \dx+\int_B  \beta\gamma_0 \left|\nabla u^{[g]}_{\gamma_0}\right|^2 \dx
\end{eqnarray*}
where $C$ is defined by \req{Def_C} in lemma~\ref{lemma:tedious_computation}.

Hence, if there exists a current pattern $g=(g_r)_{r=1}^N\in \C^N$ such that the resulting DC potential
$
u^{[g]}_{\gamma_0}
$
possesses a very large energy in $B\setminus D$ and a very small energy in $D$, then for this $g$, we can expect that
\[
g^* \left[ R((1+\beta\chi_B)\gamma_0) - \Re\left( \alpha R(\gamma_\omega)\right)\right] g <0,
\]
resp.,
\[
g^* \left[ R((1-\beta\chi_B)\gamma_0) - \Re\left( \alpha R(\gamma_\omega)\right)\right] g 
> 0,
\]
so that
\[
\Re\left(\alpha R( \gamma_\omega) \right) \not\leq R((1+\beta\chi_B)\gamma_0),
\quad
\mbox{resp.,}
\quad
\Re\left(\alpha R( \gamma_\omega) \right) \not\geq R((1-\beta\chi_B)\gamma_0).
\]
\hfill $\Box$

\section{Numerical results}\label{Sec:numerics}

In this section, we numerically demonstrate our new method for the practically relevant electrode setting of Section \ref{Subsec:setting_electrodes}. In all of the following settings, $m$ electrodes $\mathcal E_1,\mathcal E_2,...,\mathcal E_m$ are numbered as shown in the corresponding figures, and adjacent-adjacent dipole driving patterns are used according to this numbering, i.e., in the notation of section \ref{Subsec:setting_electrodes},
\[
(j_r,k_r):=(r,r+1) \quad \mbox{ for } r=1,\ldots ,m-1,\quad \mbox{ and } \quad (j_m,k_m):=(m,1).
\]
The EIT measurements at zero and non-zero frequency, and with and without ultra\-sound-mo\-du\-lation, are simulated 
by solving the equations (\ref{eq:shunt1})-(\ref{eq:shunt4}) using
MATLAB\textsuperscript{\textregistered} and the commercial FEM-software COMSOL\textsuperscript{\textregistered}.

At this point, let us stress again, that in a practical application of our new method, all required quantities are measured and no numerical simulations have to be carried out. 

\begin{exa}\label{exa:exa1}
Consider the setting illustrated in figure \ref{fig:2D_domain}. The imaging domain $\Omega$ is a two-dimensional circle with radius $10$ centered at $(0,0)$ and a circular anomaly $D$ (sketched in red in figure \ref{fig:2D_domain}) with radius $1.5$ is located
at $(5,0)$. On the boundary $\partial\Omega$, there are 16 electrodes $\mathcal E_1,\mathcal E_2,\ldots,\mathcal E_{16}$ attached.

\begin{figure}%[!h]
\begin{center}
\includegraphics[scale=1]{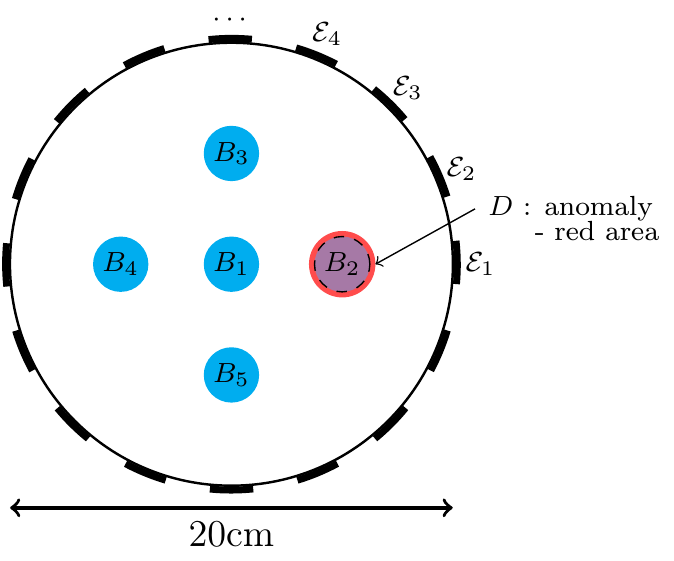}
\caption{Measurement setting of example \ref{exa:exa1}.}\label{fig:2D_domain}
\end{center}
\end{figure}

The DC and AC admittivities $\gamma_0$ and $\gamma_\omega$ are chosen as 
\begin{equation*}
\gamma_0:=1, \quad \mbox{ and } \quad \gamma_\omega:=\left\{ \begin{array}{ll}
1+\im \omega& \mbox{ in } \Omega\setminus D,\\
1+2\im \omega   & \mbox{ in } D,
\end{array}\right.
\end{equation*}
with $\omega=200\pi$, i.e. $\sigma_\Omega=\sigma_D=1$, $\epsilon_\Omega=1$, and $\epsilon_D=2$.
Hence, the ratio of the background conductivities is $\alpha=1+\im \omega$, and the contrast assumption in theorem~\ref{thm:electrode} is fulfilled with $c=\epsilon_D \sigma_\Omega -\epsilon_\Omega\sigma_D=1$.

Theorem \ref{thm:electrode} guarantees that 
\begin{equation}\labeq{main_impl_elect_cpos_numex1}
 B\subseteq D \quad \mbox{ implies that } \quad
\Re\left(\alpha  R( \gamma_\omega) \right) \leq R((1+\beta\chi_B)\gamma_0),
\end{equation}
i.e., that the ultrasound modulated DC measurements $R((1+\beta\chi_B)\gamma_0)$
are larger (in the sense of matrix definiteness) than (the real part of ratio-weighted) AC measurements $\Re\left(\alpha  R( \gamma_\omega) \right)$ if the ultrasound-modulated focusing region $B$ lies inside the inclusion $D$, and the modulation strenth $\beta>0$ is small enough. Remark \ref{remark:electrode} suggests that the converse of \req{main_impl_elect_cpos_numex1} is true if enough electrodes are used. To test this numerically, we choose 5 circular focusing regions $B_1,\ldots,B_5$ (sketched in blue in figure \ref{fig:2D_domain}) with radius 1.25. The modulation strength is chosen to be (cf. theorem \ref{thm:electrode}) 
\[
\beta=\omega^2 |c| \frac{\epsilon_\Omega}{\sigma_D(\sigma_\Omega^2 + \omega^2 \epsilon_\Omega^2)}\approx 0.9999.
\]

Table \ref{table:exa1} shows the eigenvalues of $R((1+\beta\chi_{B_j})\gamma_0)-\Re\left(\alpha R( \gamma_\omega) \right)$ for $j\in\lbrace1,\cdots,5\rbrace$. The numerical error ($\delta\approx 0.0005$) in table \ref{table:exa1} was estimated by repeating the calculations on a finer FEM grid. 
Taking into account this estimated numerical error, the monotonicity test $R((1+\beta\chi_{B_j})\gamma_0)\geq \Re\left(\alpha R( \gamma_\omega) \right)$ is only fulfilled for the focusing region $B_2$, which lies inside the inclusion. 
\begin{table}
\centering
\scalebox{0.9}{\begin{tabular}{c|c|c|c|c}
\hline
$B_1$ & $B_2$ & $B_3$& $B_4$ & $B_5$\\ \hline
   -0.0024 $\pm$ 0.0005 &   0.0000 $\pm$ 0.0005 &   -0.0098 $\pm$ 0.0005 &   -0.0105 $\pm$ 0.0005 &   -0.0098 $\pm$ 0.0005\\
   -0.0024 $\pm$ 0.0005 &   0.0000 $\pm$ 0.0005 &   -0.0097 $\pm$ 0.0005 &   -0.0104 $\pm$ 0.0005 &   -0.0097 $\pm$ 0.0005\\
   -0.0000 $\pm$ 0.0005 &   0.0000 $\pm$ 0.0005 &   -0.0003 $\pm$ 0.0005 &   -0.0004 $\pm$ 0.0005 &   -0.0003 $\pm$ 0.0005\\
   -0.0000 $\pm$ 0.0005 &   0.0000 $\pm$ 0.0005 &   -0.0002 $\pm$ 0.0005 &   -0.0003 $\pm$ 0.0005 &   -0.0002 $\pm$ 0.0005 \\
   -0.0000 $\pm$ 0.0005 &   0.0000 $\pm$ 0.0005 &   -0.0000 $\pm$ 0.0005 &   -0.0000 $\pm$ 0.0005 &   -0.0000 $\pm$ 0.0005 \\
   -0.0000 $\pm$ 0.0005 &   0.0000 $\pm$ 0.0005 &   -0.0000 $\pm$ 0.0005 &   -0.0000 $\pm$ 0.0005 &   -0.0000 $\pm$ 0.0005 \\
   -0.0000 $\pm$ 0.0005 &   0.0000 $\pm$ 0.0005 &   -0.0000 $\pm$ 0.0005 &   -0.0000 $\pm$ 0.0005 &   -0.0000 $\pm$ 0.0005 \\
   -0.0000 $\pm$ 0.0005 &   0.0000 $\pm$ 0.0005 &   -0.0000 $\pm$ 0.0005 &   -0.0000 $\pm$ 0.0005 &   -0.0000 $\pm$ 0.0005 \\
   -0.0000 $\pm$ 0.0005 &   0.0000 $\pm$ 0.0005 &   -0.0000 $\pm$ 0.0005 &   -0.0000 $\pm$ 0.0005 &   -0.0000 $\pm$ 0.0005 \\
   -0.0000 $\pm$ 0.0005 &   0.0000 $\pm$ 0.0005 &   -0.0000 $\pm$ 0.0005 &    0.0000 $\pm$ 0.0005 &   -0.0000 $\pm$ 0.0005 \\
    0.0000 $\pm$ 0.0005 &   0.0000 $\pm$ 0.0005 &    0.0000 $\pm$ 0.0005 &    0.0000 $\pm$ 0.0005 &    0.0000 $\pm$ 0.0005 \\
    0.0000 $\pm$ 0.0005 &   0.0000 $\pm$ 0.0005 &    0.0000 $\pm$ 0.0005 &    0.0000 $\pm$ 0.0005 &    0.0000 $\pm$ 0.0005 \\
    0.0003 $\pm$ 0.0005 &   0.0004 $\pm$ 0.0005 &    0.0005 $\pm$ 0.0005 &    0.0007 $\pm$ 0.0005 &    0.0005 $\pm$ 0.0005 \\
    0.0006 $\pm$ 0.0005 &   0.0005 $\pm$ 0.0005 &    0.0007 $\pm$ 0.0005 &    0.0009 $\pm$ 0.0005 &    0.0007 $\pm$ 0.0005 \\
    0.0142 $\pm$ 0.0005 &   0.0048 $\pm$ 0.0005 &    0.0146 $\pm$ 0.0005 &    0.0153 $\pm$ 0.0005 &    0.0146 $\pm$ 0.0005 \\
    0.0145 $\pm$ 0.0005 &   0.0049 $\pm$ 0.0005 &    0.0148 $\pm$ 0.0005 &    0.0155 $\pm$ 0.0005 &    0.0148 $\pm$ 0.0005 \\\hline
\end{tabular}}
\caption{Eigenvalues of $R((1+\beta\chi_{B_j})\gamma_0)-\Re\left(\alpha R( \gamma_\omega) \right)$, $j=1,\ldots,5$, for example \ref{exa:exa1}.}\label{table:exa1}
\end{table}
\end{exa}

\begin{exa}\label{exa:exa2}
Now we consider the three-dimensional setting illustrated in Figure \ref{fig:3D_domain}. The imaging domain $\Omega$ is a cylindrical domain with
\[
\Omega=\left\lbrace (x_1,x_2,x_3)\in\mathbb{R}^3\,:\|(x_1,x_2,0)\|<10,\ 0<x_3<5\right\rbrace.
\]
and a ball-shaped anomaly $D$ with radius $1.5$ is located at $(5,0,2.5)$. On the boundary $\partial\Omega$, there are 16 electrodes $\mathcal E_1,\mathcal E_2,\ldots,\mathcal E_{16}$ attached.

\begin{figure}%[!h]
\begin{minipage}{.45\textwidth}
 \centering
  \includegraphics[scale=0.85]{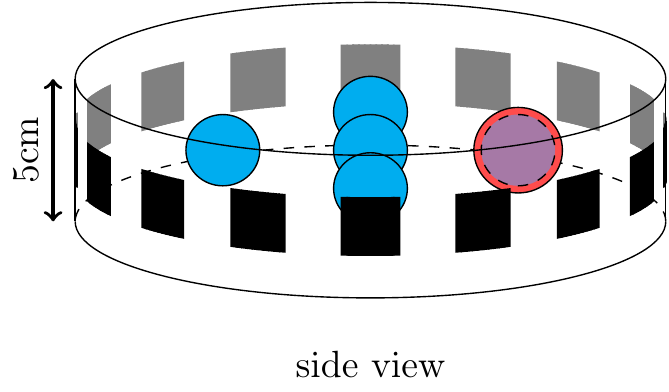}
\end{minipage}
\hspace*{0.05\textwidth}
\begin{minipage}{.45\textwidth}
 \centering
  \includegraphics[scale=0.85]{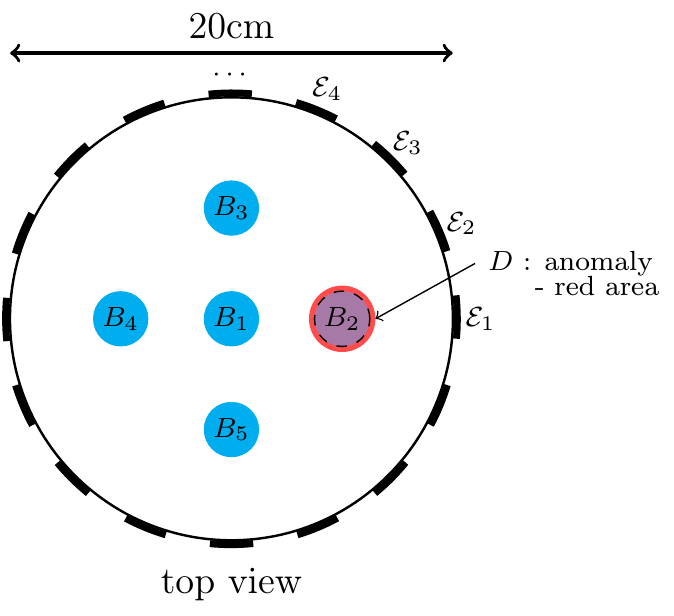}
\end{minipage}
\caption{Measurement setting of example \ref{exa:exa2}.}\label{fig:3D_domain}
\end{figure}

The DC and AC admittivities $\gamma_0$ and $\gamma_\omega$ are chosen as 
\begin{equation*}
\gamma_0:=\left\{ \begin{array}{ll}
1, & \mbox{ in } \Omega\setminus D,\\
2, & \mbox{ in } D,
\end{array}\right.
 \quad \mbox{ and } \quad \gamma_\omega:=\left\{ \begin{array}{ll}
1+\im \omega, & \mbox{ in } \Omega\setminus D,\\
2+\im \omega, & \mbox{ in } D,
\end{array}\right.
\end{equation*}
with $\omega=200\pi$, %i.e. $\sigma_\Omega=1$, $\sigma_D=2$, $\epsilon_\Omega=1=\epsilon_D=$.
so that $\alpha=1+\im \omega$ and $c=-1$. As in example \ref{exa:exa1} we check the monotonicity relation for five focusing regions $B_1,B_2,B_3,B_4$ and $B_5$. The regions are ball-shaped with radius $1.25$ and centered at
$(0,0,2.5),(5,0,2.5),(0,5,2.5),(-5,0,2.5)$ and $(0,-5,2.5)$, respectively.
We choose $\beta$ according to theorem \ref{thm:electrode} as 
\[
\beta=\omega^2 |c|   \frac{\epsilon_D}{\sigma_D (\sigma_D \sigma_\Omega + \omega^2 \epsilon_D \epsilon_\Omega)}\approx 0.4999.
\]

Table \ref{table:exa2} shows the eigenvalues of $R((1-\beta\chi_{B_j})\gamma_0)-\Re\left(\alpha R( \gamma_\omega) \right)$ for $j=\lbrace1,\cdots,5\rbrace$. The numerical error ($\delta\approx 0.14$) in table \ref{table:exa2} was estimated by repeating the calculations on a finer FEM grid.  
Taking into account this estimated numerical error, the monotonicity test $R((1-\beta\chi_{B_j})\gamma_0)\leq \Re\left(\alpha R( \gamma_\omega) \right)$ is only fulfilled for the second focussing region, which lies inside the inclusion.
\begin{table}
\centering
\scalebox{0.9}{\begin{tabular}{c|c|c|c|c}
\hline
$B_1$ & $B_2$ & $B_3$& $B_4$ & $B_5$\\ \hline
   -0.1639 $\pm$ 0.0136&   -0.0800 $\pm$ 0.0136&   -0.1667 $\pm$ 0.0136&   -0.1723 $\pm$ 0.0136&   -0.1668 $\pm$ 0.0136\\
   -0.1621 $\pm$ 0.0136&   -0.0785 $\pm$ 0.0136&   -0.1642 $\pm$ 0.0136&   -0.1696 $\pm$ 0.0136&   -0.1642 $\pm$ 0.0136\\
   -0.0052 $\pm$ 0.0136&   -0.0054 $\pm$ 0.0136&   -0.0065 $\pm$ 0.0136&   -0.0079 $\pm$ 0.0136&   -0.0065 $\pm$ 0.0136\\
   -0.0034 $\pm$ 0.0136&   -0.0041 $\pm$ 0.0136&   -0.0045 $\pm$ 0.0136&   -0.0060 $\pm$ 0.0136&   -0.0045 $\pm$ 0.0136\\
   -0.0001 $\pm$ 0.0136&   -0.0003 $\pm$ 0.0136&   -0.0002 $\pm$ 0.0136&   -0.0003 $\pm$ 0.0136&   -0.0002 $\pm$ 0.0136\\
   -0.0000 $\pm$ 0.0136&   -0.0001 $\pm$ 0.0136&   -0.0000 $\pm$ 0.0136&   -0.0001 $\pm$ 0.0136&   -0.0000 $\pm$ 0.0136\\
   -0.0000 $\pm$ 0.0136&   -0.0000 $\pm$ 0.0136&   -0.0000 $\pm$ 0.0136&   -0.0000 $\pm$ 0.0136&   -0.0000 $\pm$ 0.0136\\
   -0.0000 $\pm$ 0.0136&   -0.0000 $\pm$ 0.0136&   -0.0000 $\pm$ 0.0136&   -0.0000 $\pm$ 0.0136&   -0.0000 $\pm$ 0.0136\\
   -0.0000 $\pm$ 0.0136&   -0.0000 $\pm$ 0.0136&   -0.0000 $\pm$ 0.0136&    0.0000 $\pm$ 0.0136&   -0.0000 $\pm$ 0.0136\\
    0.0000 $\pm$ 0.0136&   -0.0000 $\pm$ 0.0136&    0.0000 $\pm$ 0.0136&    0.0000 $\pm$ 0.0136&    0.0000 $\pm$ 0.0136\\
    0.0000 $\pm$ 0.0136&   -0.0000 $\pm$ 0.0136&    0.0000 $\pm$ 0.0136&    0.0000 $\pm$ 0.0136&    0.0000 $\pm$ 0.0136\\
    0.0000 $\pm$ 0.0136&   -0.0000 $\pm$ 0.0136&    0.0000 $\pm$ 0.0136&    0.0001 $\pm$ 0.0136&    0.0000 $\pm$ 0.0136\\
    0.0001 $\pm$ 0.0136&   -0.0000 $\pm$ 0.0136&    0.0011 $\pm$ 0.0136&    0.0019 $\pm$ 0.0136&    0.0011 $\pm$ 0.0136\\
    0.0001 $\pm$ 0.0136&    0.0000 $\pm$ 0.0136&    0.0017 $\pm$ 0.0136&    0.0025 $\pm$ 0.0136&    0.0017 $\pm$ 0.0136\\
    0.0169 $\pm$ 0.0136&    0.0000 $\pm$ 0.0136&    0.0728 $\pm$ 0.0136&    0.0788 $\pm$ 0.0136&    0.0717 $\pm$ 0.0136\\
    0.0173 $\pm$ 0.0136&    0.0000 $\pm$ 0.0136&    0.0749 $\pm$ 0.0136&    0.0822 $\pm$ 0.0136&    0.0753 $\pm$ 0.0136\\
\hline
\end{tabular}}\caption{Eigenvalues of $R((1-\beta\chi_{B_j})\gamma_0)- \Re\left(\alpha R( \gamma_\omega) \right)$, $j=1,\ldots,5$, for example \ref{exa:exa2}.}
\label{table:exa2}
\end{table}
\end{exa}

\begin{exa}\label{exa:exa3}
In our last example we test a large number or small balls in order to demonstrate up to which extend 
the method is capable of determining the shape of an inclusion. We consider the two- and three-dimensional 
example shown in figure \ref{fig:setting_rec_2D_combined}, and \ref{fig:rec_3D_setting_combined}, respectively.
In both settings,
\begin{equation*}
\gamma_0:=1,  \quad \mbox{ and } \quad
\gamma_\omega:=\left\{ \begin{array}{ll}
1+2 \im \omega& \mbox{in}~\Omega\backslash D,\\
1+\im \omega  & \mbox{in}~D,
\end{array}\right.
\end{equation*}
with $\omega=200\pi$, %i.e. $\sigma_\Omega=1$, $\sigma_D=2$, $\epsilon_\Omega=1=\epsilon_D=$.
so that $\alpha=1+2\im \omega$ and $c=-1$. In accordance with theorem \ref{thm:electrode}, we choose 
\[
\beta=\omega^2 |c|   \frac{\epsilon_D}{\sigma_D (\sigma_D \sigma_\Omega + \omega^2 \epsilon_D \epsilon_\Omega)}\approx 0.4999.
\] 

\begin{figure}
 \centering
 \includegraphics[scale=1]{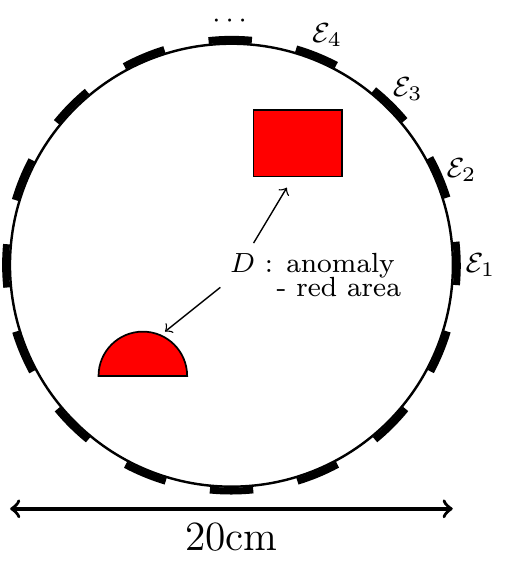}
 \caption{Two-dimensional measurement setting of example \ref{exa:exa3}.}
 \label{fig:setting_rec_2D_combined}
\end{figure}

\begin{figure}
\begin{minipage}{.45\textwidth}
 \centering
  \includegraphics[scale=0.9]{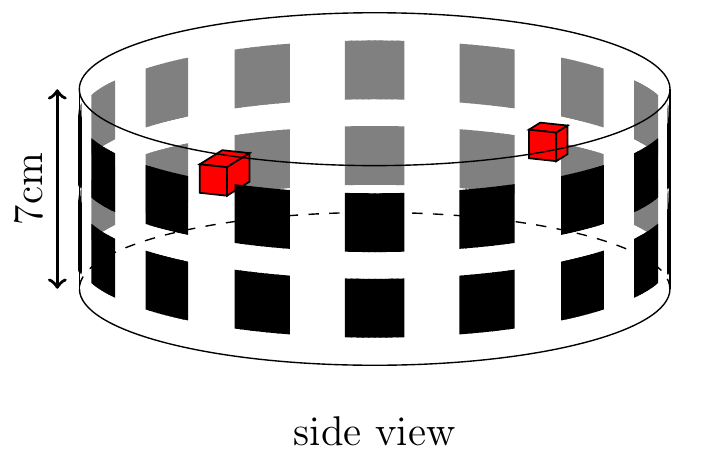}
\end{minipage}
\hspace*{0.05\textwidth}
\begin{minipage}{.45\textwidth}
 \centering
  \includegraphics[scale=0.9]{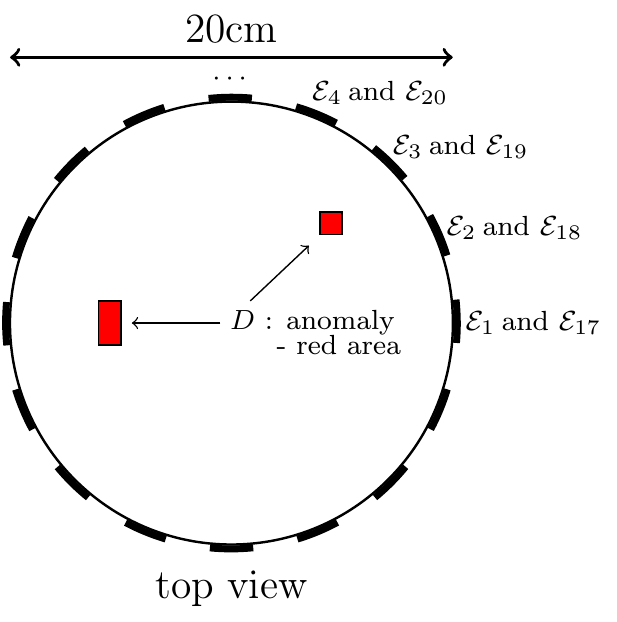}
\end{minipage}
\caption{Three-dimensional measurement setting of example \ref{exa:exa3}.}
\label{fig:rec_3D_setting_combined}
\end{figure}

We now consider a large number of test balls $B_j$, $j\in\lbrace 1,2,\ldots,N\rbrace$, and mark all balls for which
\begin{equation}\labeq{num_recon_test}
R((1-\beta\chi_{B_j})\gamma_0) - \Re\left(\alpha R( \gamma_\omega) \right) \leq \delta I,
\end{equation}
where $I$ is the identity matrix and $\delta>0$ is a regularization parameter. In both examples, 
we used the heuristically chosen value $\delta=0.5\cdot 10^{-7}$. Figure \ref{fig:rec_2D_combined} and figure \ref{fig:rec_3D_combined} show the test balls (in blue), the true inclusion (in red) and the balls for which \req{num_recon_test} is fulfilled (in grey). 

\begin{figure}
\begin{minipage}{.45\textwidth}
 \centering
 \includegraphics[trim = 137 222 97 257,clip,scale=0.35]{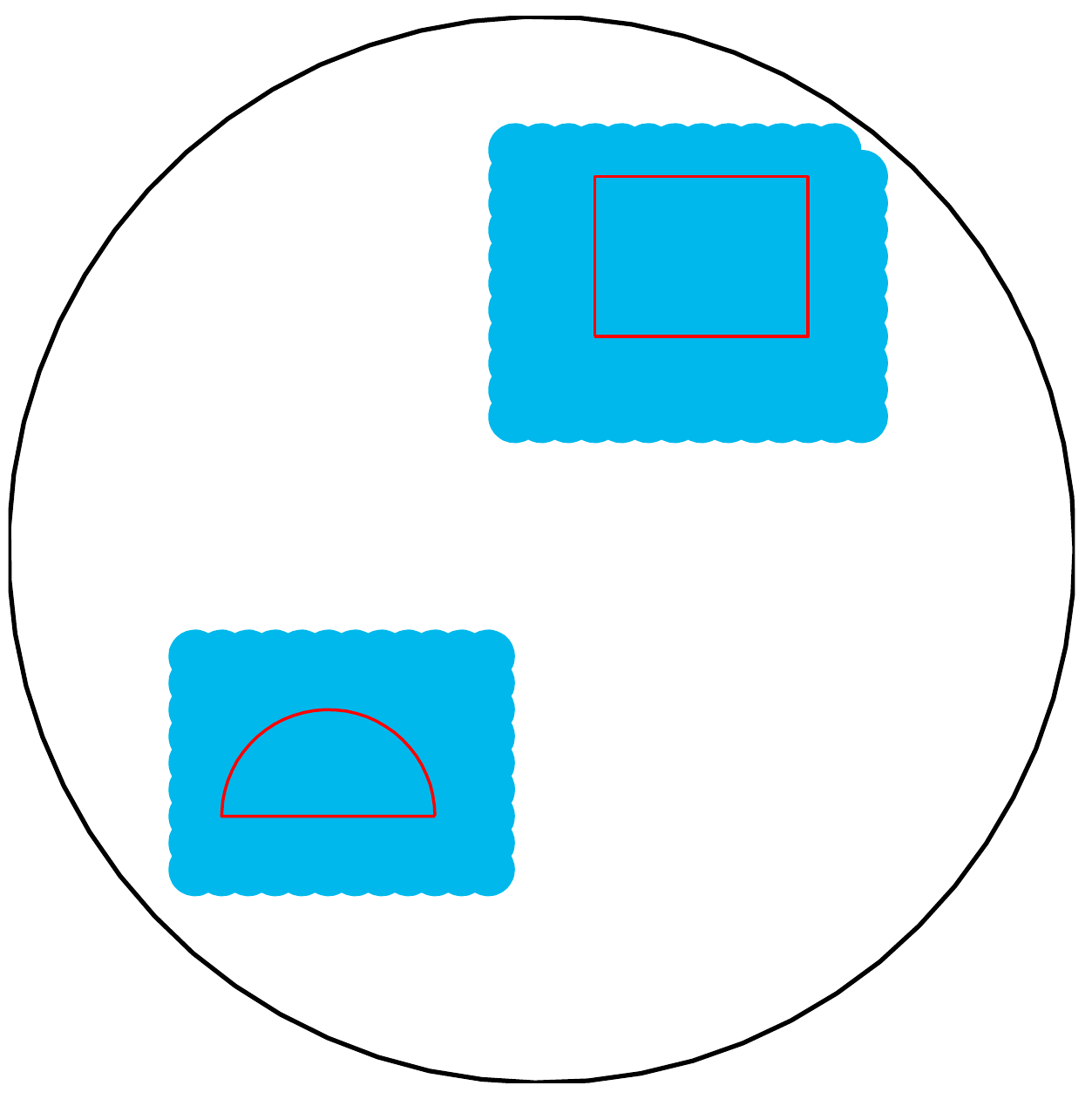}
\end{minipage}
\hspace*{0.05\textwidth}
\begin{minipage}{.45\textwidth}
 \centering
 \includegraphics[trim = 137 222 97 257,clip,scale=0.35]{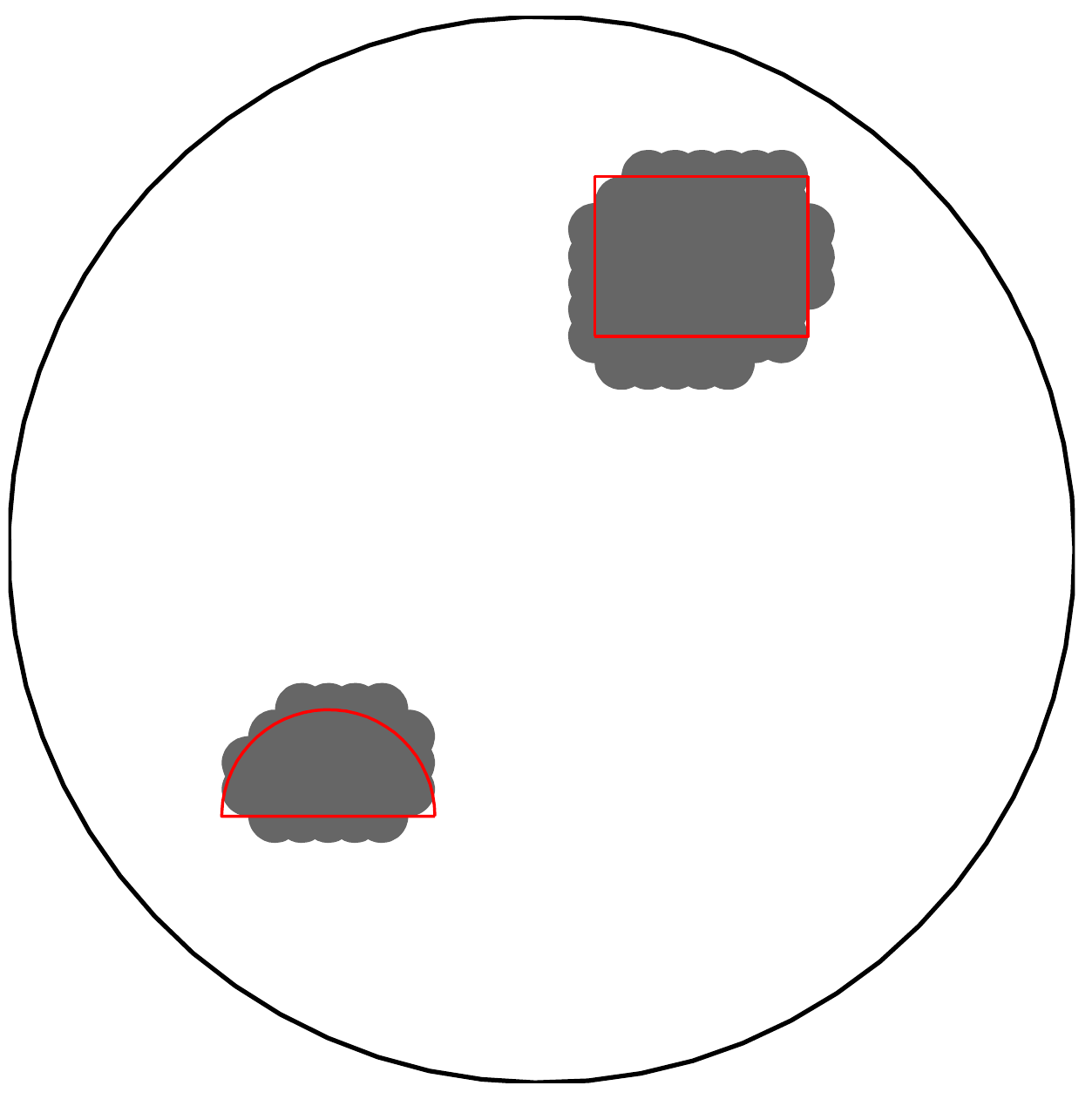}
\end{minipage} 
 \caption{Results for the two-dimensional setting in example \ref{exa:exa3}.}\label{fig:rec_2D_combined}
\end{figure}

\begin{figure}
 \centering
 \includegraphics[trim = 130 320 95 290,clip,scale=0.85]{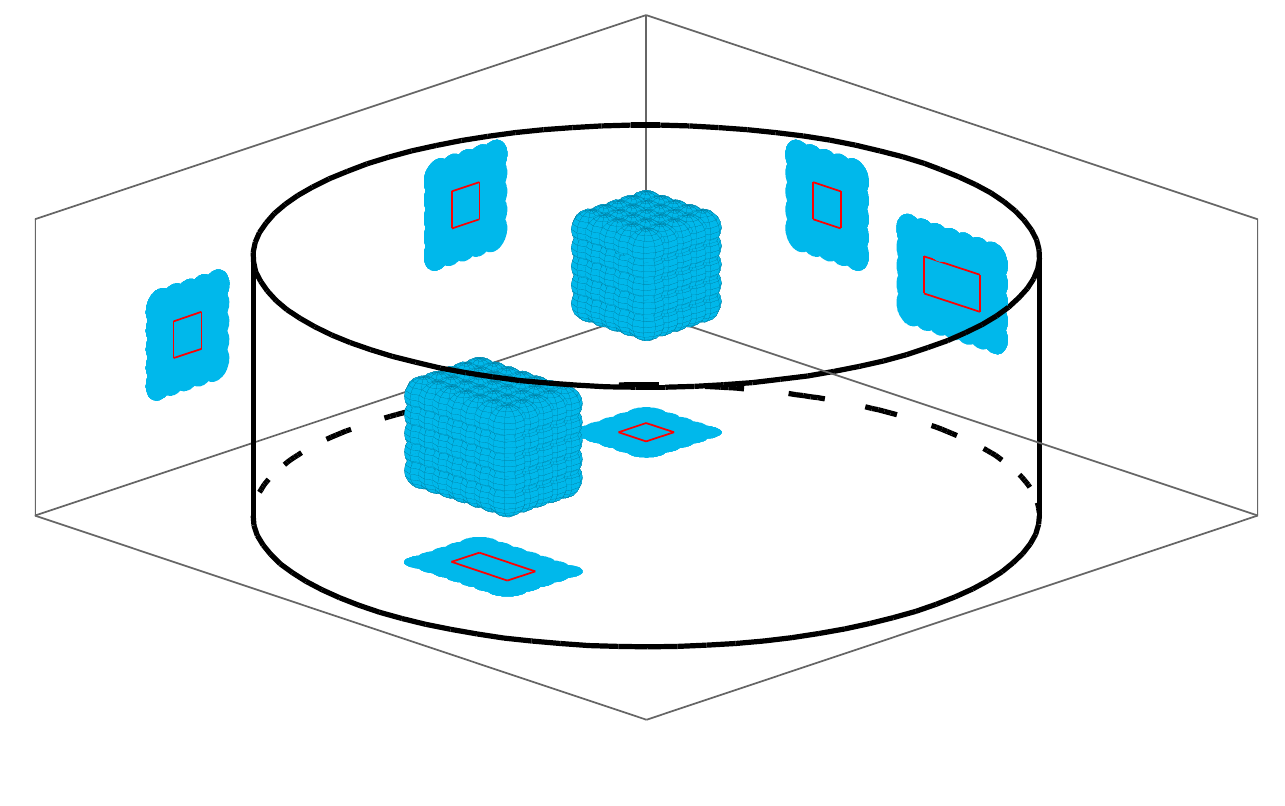}
 
 \vspace*{-1.5em}
 
 \includegraphics[trim = 130 320 95 290,clip,scale=0.85]{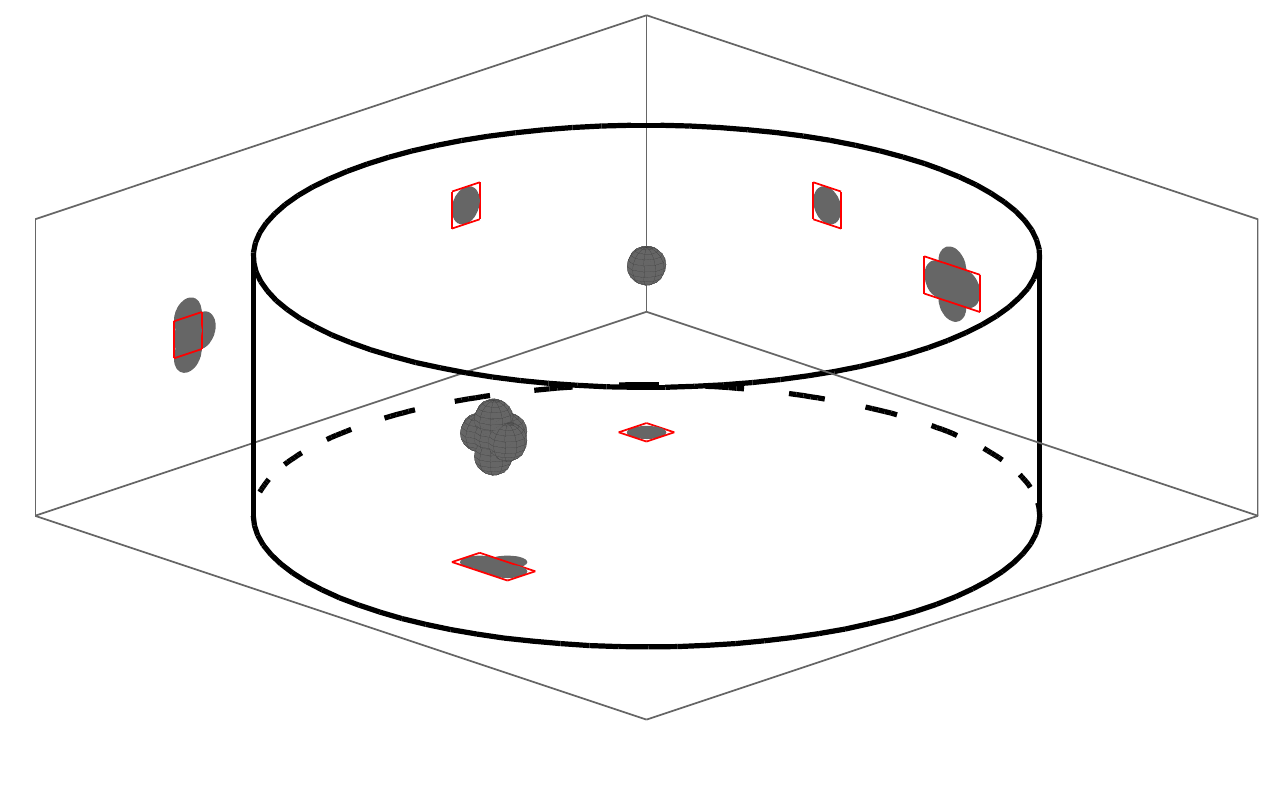}

% \vspace*{-1.5em}
 
% \includegraphics[trim = 130 320 95 290,clip,scale=0.85]{rec_3D_tol_2_grey.pdf}

 \caption{Results for the three-dimensional setting in example \ref{exa:exa3}.}\label{fig:rec_3D_combined}
\end{figure}
\end{exa}

\section{Conclusion and discussion}\label{Sec:conlusion_and_discussion}
We have developed a new method to detect and localize conductivity anomalies by combining frequency-difference electrical impedance tomography (EIT) with ultrasound-modulated EIT. Our method is based on comparing (in terms of matrix definiteness) ultrasound-modulated EIT measurements with (the real part of ratio-weighted) EIT measurements at a non-zero frequency. We showed that this comparison determines whether the focusing region of the ultrasound wave lies inside a conductivity anomaly or not. 

Remarkably, our new method merely utilizes the two sets of EIT measurements, and the background conductivity ratio which in turn can be estimated from EIT measurements. The method does not require any numerical simulations, forward calculations or geometry-dependent special solutions. It can be implemented without knowing the imaging domain shape or the electrode position, and is thus completely unaffected by modeling errors. 

We gave a rigorous mathematical proof for our new method for the case of continuous boundary data, and we justified why the method can be expected to work also for realistic electrode measurements, provided that the number of electrodes is large enough. 

The method is based on the assumption that the background conductivity is spatially constant, and that the anomalies fulfill the contrast condition that is required in frequency-difference EIT. 
Also, our method relies on the idealistic assumption of ultrasound modulated EIT, that it is possible to perfectly focus an ultrasound wave so that the conductivity changes only in a small test region. In real applications, background conductivities can be expected to be at least slightly inhomogeneous, and the ultrasound wave will also have some effect on the conductivity outside the focusing region. 
The performance of our new method in such a setting has yet to be evaluated. Let us however note that the matrix definiteness comparisons, that are used by our method, are principally stable (cf.\ remark \ref{rem:stable}) so that our arguably idealistic modeling assumptions only have to be approximately
valid. Moreover, monotonicity arguments also allow for worst-case testing and resolution guarantees (cf.\ \cite{harrachresolution}) which might be helpful in relaxing the idealistic assumptions in future studies.

\ack
BH and MU would like to thank the German Research
Foundation (DFG) for financial support of the project within
the Cluster of Excellence in Simulation Technology (EXC
310/1) at the University of Stuttgart.

\section*{References}

\bibliography{literaturliste}
\bibliographystyle{abbrv}

\end{document}